\documentclass[a4paper,reqno]{amsart}

\usepackage[T1]{fontenc}
\usepackage[utf8]{inputenc}

\usepackage[unicode,colorlinks=true,linkcolor=blue,urlcolor=cyan,citecolor=red,anchorcolor=blue]{hyperref}

\usepackage{amsmath,amssymb,amsthm,mathtools}
\usepackage{cite}
\usepackage{tikz, tikz-cd, stmaryrd}

\usepackage{enumitem}
\setlist[enumerate]{label = \textup{\textup{(\alph*)}},ref = \textup{(\alph*})}

\theoremstyle{plain}
\newtheorem{thm}{Theorem}[section]

\newtheorem{prop}[thm]{Proposition}
\newtheorem{cor}[thm]{Corollary}

\theoremstyle{definition}
\newtheorem{dfn}[thm]{Definition}
\newtheorem{ex}[thm]{Example}
\newtheorem{rem}[thm]{Remark}

\numberwithin{equation}{section}
\numberwithin{figure}{section}
\numberwithin{table}{section}

\def\N{\mathbb{N}}
\def\Z{\mathbb{Z}}

\def\R{\mathbb{R}}
\def\C{\mathbb{C}}
\def\H{\mathbb{H}}
\def\K{\mathbb{K}}

\DeclareMathOperator\id{id}

\DeclareMathOperator\ev{ev}

\DeclareMathOperator\Hom{Hom}
\DeclareMathOperator\End{End}
\DeclareMathOperator\Aut{Aut}

\let\Im\relax\DeclareMathOperator\Im{Im}

\DeclareMathOperator\Ker{Ker}
\DeclareMathOperator\Coker{Coker}
\DeclareMathOperator\Spec{Spec}

\def\O{\mathrm{O}}
\def\SO{\mathrm{SO}}
\def\Spin{\mathrm{Spin}}
\def\U{\mathrm{U}}
\def\SU{\mathrm{SU}}

\def\g{{\mathfrak g}}
\def\h{{\mathfrak h}}

\def\Ad{\mathrm{Ad}}

\def\cA{\mathcal{A}}

\def\cD{\mathcal{D}}

\def\cJ{\mathcal{J}}

\def\cR{\mathcal{R}}

\def\cV{\mathcal{V}}

\def\al{\alpha}
\def\be{\beta}
\def\ga{\gamma}
\def\de{\delta}

\def\io{\iota}

\def\la{\lambda}
\def\si{\sigma}

\def\om{\omega}

\def\De{\Delta}

\def\La{\Lambda}

\def\Ga{\Gamma}
\def\Si{\Sigma}

\def\op{\oplus}
\def\ot{\otimes}
\def\t{\times}

\def\6{\partial}
\def\d{{\rm d}}
\def\na{\nabla}

\def\iy{\infty}
\def\longra{\longrightarrow}

\def\an#1{\langle #1 \rangle}
\def\ban#1{\bigl\langle #1 \bigr\rangle}

\def\ol{\overline}

\def\es{\emptyset}
\def\ab{\allowbreak}

\def\bs{\boldsymbol}

\def\ge{\geqslant}
\def\le{\leqslant}

\def\e#1\e{\begin{equation}#1\end{equation}}

\makeatletter
\newcommand{\ea}{\@ifstar{\ea@star}{\ea@unstar}}
\def\ea@star#1\ea{\@ifstar{\ea@star@star{#1}}{\ea@star@unstar{#1}}}
\def\ea@unstar#1\ea{\@ifstar{\ea@unstar@star{#1}}{\ea@unstar@unstar{#1}}}

\newcommand{\ea@unstar@unstar}[1]{\begin{align}#1\end{align}}
\newcommand{\ea@star@star}[1]{\begin{align*}#1\end{align*}}
\newcommand{\ea@unstar@star}[1]{\PackageError{preamble}{Unstarred macro invoked but macro ends with a star}{There are two options: You may either put a star before and after ea (this creates an unnumbered align environment) or use no star at all (this creates a numbered align environment)}
}
\newcommand{\ea@star@unstar}[1]{\PackageError{preamble}{Starred macro invoked but macro ends with no star}{There are two options: You may either put a star before and after ea (this creates an unnumbered align environment) or use no star at all (this creates a numbered align environment)}
}
\makeatother

\makeatletter
\ifcsname phantomsection\endcsname
    \newcommand*{\qrr@gobblenexttocentry}[5]{}
\else
    \newcommand*{\qrr@gobblenexttocentry}[4]{}
\fi
\newcommand*{\addsubsection}{%
    \addtocontents{toc}{\protect\qrr@gobblenexttocentry}%
    \subsection}
\makeatother

\usepackage{amsrefs}
\setlist[enumerate]{label = \textbf{\textup{\textup{(\alph*)}}},ref = \textup{(\alph*})}

\usepackage[english]{babel}
\usepackage{slashed}

\usepackage{stmaryrd}
\SetSymbolFont{stmry}{bold}{U}{stmry}{m}{n}

\usepackage{array}
\usepackage{caption}
\usepackage{subcaption}

\def\Eig{\operatorname{Eig}}

\def\adj{\mathrm{adj}}
\def\APS{\mathrm{APS}}
\def\stab{\si}
\def\trans{\mathrm{trans}}

\def\PF{\mathop{\rm PF}\nolimits}
\def\DET{\mathop{\rm DET}\nolimits}
\def\SP{\mathop{\rm SP}\nolimits}

\def\ind{\mathop{\rm ind}\nolimits}

\def\skew{\mathrm{skew}}
\def\sym{\mathrm{sym}}

\def\boo{{\mathbin{\mathbf 1}}}

\def\m{{\mathfrak m}}
\def\KOtor{\Ga_\ell/\!\!/\Ga_{\ell+1}}
\def\Bord{{\mathfrak{Bord}}{}}
\def\rO{\mathrm{O}}

\def\Cl{C\ell}

\def\Cyl{\operatorname{Cyl}}
\def\Ell{\operatorname{Ell}}
\def\Lag{\operatorname{Lag}}

\def\det{\operatorname{det}}

\begin{document}

\title[Bordism invariance of orientations and real APS index theory]{Bordism invariance of orientations\newline and real APS index theory}

\author{Markus Upmeier}

\address{Department of Mathematics, University of Aberdeen, Fraser Noble Building, Elphinstone Rd, Aberdeen, AB24 3UE, U.K.}

\email{markus.upmeier@abdn.ac.uk}

\date{\today}

\begin{abstract}
We show that orientations and Floer gradings for elliptic differential operators can be propagated through bordisms. This is based on a new perspective on APS indices for elliptic boundary value problems over the real numbers. Several applications to moduli spaces of this new bordism-theoretic point of view will be given in the sequel.
\end{abstract}

\thanks{This research was partly funded by a Simons Collaboration Grant on `Special Holonomy in Geometry, Analysis and Physics'. The author thanks Dominic Joyce for many discussions.}



\maketitle

\setcounter{tocdepth}{2}
\tableofcontents

\section{Introduction}
\label{bio1}

In the study of moduli spaces of solutions to elliptic differential equations, orientations and Floer gradings are an important ingredient for the construction of enumerative invariants \cites{DoSe,DoTh,DoKr,Joyc5,Walp1}. Examples include moduli spaces of instantons or moduli spaces of calibrated submanifolds. Constructing orientations is often a challenging problem due to the lack of a theoretical framework. This paper and the sequels \cites{JUother,JUfurther} develop a new such framework based on bordism theory.

The present paper develops the analytic foundation of this new bordism-theoretic perspective and exhibits a new relationship between `categorical indices' and certain elliptic boundary conditions. Recall that the bordism-invariance of the index in $KO$-theory for real Dirac operators is a consequence of the Atiyah--Singer Index Theorem \cites{AtSi4,AtSi5}. The main result of this paper, Theorem~\ref{bio3thm1}, can be viewed as a higher-categorical rendition of this classical result.

Motivated by the properties of the real Dirac operator in dimension $\ell$ we introduce the notion of an `$\ell$-adapted' operator; in general, $\ell$ may be different from the dimension. Problems involving twisted Dirac operators often give $\ell$-adapted operators. We then construct the \emph{orientation torsor} $\rO_\ell(D_X)$ of an `$\ell$-adapted' elliptic differential operator $D_X$ on a compact manifold without boundary, which encodes orientations of $\Ker(D_X),$ $\Coker(D_X),$ and Floer gradings of the spectrum of $D_X.$ The orientation torsor is a kind of categorification of the classical index. Our main result shows that every 
$(\ell+1)$-adapted elliptic differential operator $D_W$ on a bordism $W$ from $(X_0,D_{X_0})$ to $(X_1,D_{X_1})$ induces a canonical isomorphism $\rO_\ell[D_W]\colon\rO_\ell(D_{X_0})\to\rO_\ell(D_{X_1})$ of orientation torsors. In this sense a \emph{categorical symmetry} for orientation torsors is established, where the symmetry is described not by a group, but by a category (the elliptic bordism category $\Bord_n^{\Ell_\ell}$ defined below). For example, from our point of view the action of the gauge group on the orientations on moduli spaces of connections as in Donaldson--Kronheimer~\cite{DoKr} becomes the isomorphism induced by a cylinder bordism, where we use the gauge transformation to define the boundary identifications. As will be shown in Joyce--Upmeier~\cites{JUother, JUfurther} our method is very powerful for constructing orientations because it allows the passage from a manifold to a simpler, bordant manifold. We will construct coherent orientations for DT4-invariants of Calabi--Yau 4-folds, thereby settling (affirmatively) an important open question for Donaldson--Thomas type invariants \cite{DoTh}. In dimensions 7 and 8, we will also give new applications to orientations on the moduli space of Cayley submanifolds and to Floer gradings on the moduli space of $G_2$-instantons.

The main idea behind the categorical symmetry is a new approach concerning index theory of real elliptic boundary value problems. The classical papers (Atiyah--Patodi--Singer~\cite{AtPaSi1} and Lockhart--McOwen \cite{LoMc}) construct indices of elliptic boundary value problems with prescribed decay rates near the boundary, leading to an index that depends on the choice of decay rate. From our perspective on APS indices for elliptic boundary value problems, the orientation torsor allows us to construct an invariantly defined APS index that is independent of choices. We introduce the concept of a `nearly APS boundary condition' and show that an orientation on the boundary fixes an elliptic boundary condition up to deformation.

Thus, the index of an elliptic differential operator on a manifold \emph{with boundary} should be regarded not as an element of an abelian group (a `number'), but rather as an element of a \emph{torsor} constructed from the induced differential operator \emph{on} the boundary. From this perspective a choice of boundary condition determines a trivialization of the torsor and changes of trivialization agree with the change in index when passing to a different boundary condition. 

\section{The elliptic bordism category}
\label{bio2}

This section introduces a variant of the idea of a bordism category, where tangential structures are replaced by `adapted' elliptic differential operators, defined below. The key point is to specify a standard `cylindrical' form of the adapted differential operator over a collar neighborhood of the boundary of the bordism.

We also define here the orientation bundle $\rO_\ell(\cD_X)$ of a family of $\ell$-adapted elliptic differential operators on a compact manifold without boundary.

These definitions are used in the statements of our main results in \S\ref{bio3}.

\subsection{\texorpdfstring{$\ell$-adapted differential operators}{ℓ-adapted differential operators}}
\label{bio21}

Throughout, use the following notation: `$\equiv$' denotes congruence mod $8$;
$\ol V$ is the conjugate of a complex vector space $V$;
If $j\colon V\to\ol{V}$ is a quaternionic structure on $V,$ let $V^\diamond$ be the same complex vector space $V$ but with the quaternionic structure $j^\diamond=-j.$


\begin{table}[htb]
\centerline{\begin{tabular}{|l|c|c|c|c|c|c|c|c|c|c|c|}
\hline
$\bs\ell\equiv$ & $\bs 0$ & $\bs 1$ & $\bs 2$ & $\bs 3$ & $\bs 4$ & $\bs 5$ & $\bs 6$ & $\bs 7$\\\hline
$\K$ & $\R$ & $\R$ & $\C$ & $\H$ & $\H$ & $\H$ & $\C$ & $\R$\\\hline
$\Ga_\ell$ & $\Z$ & $\Z_2$ & $\Z_2$ & 0 & $\Z$ & 0 & 0 & 0\\\hline
\end{tabular}}
\caption{$\K$ is the natural base (skew) field of the real Clifford algebra $\Cl_{\ell-1}$ and $\Ga_\ell$ is the coefficient group $KO_\ell(\mathrm{pt})$ of  K-theory.}
\label{bio2tab1}
\end{table}

\begin{dfn}
\label{bio2def1}
Let $E_0, E_1\to X$ be vector bundles over a Riemannian $n$-manifold with fiberwise metrics. Let $\ell\in\N.$ A first order elliptic differential operator 
\[
 D\colon\Ga^\iy(E_0)\longra\Ga^\iy(E_1)
\]
is called \emph{$\ell$-adapted} if $E_0,$ $E_1$ have metric $\K$-linear structures, where the (skew) field $\K$ is defined according to Table~\ref{bio2tab1} and, in addition, the following conditions hold:
\begin{itemize}
\item
If $\ell\equiv 1,$ then $E_0=E_1$ and $D$ is $\R$-linear formally skew-adjoint, $D^*=-D.$
\item
If $\ell\equiv 2,$ then $E_0=\ol{E_1}$ and $D$ is $\C$-linear formally skew-adjoint, $D^*=-\ol{D}.$
\item
If $\ell\equiv 3,7,$ then $E_0=E_1$ and $D$ is $\K$-linear formally self-adjoint, $D^*=D.$
\item
If $\ell\equiv 5,$ then $E_0=E_1^\diamond$ and $D$ is $\H$-linear formally self-adjoint, $D^*=D^\diamond.$
\item
If $\ell\equiv 6,$ then $E_0=\ol{E_1}$ and $D$ is $\C$-linear formally self-adjoint, $D^*=\ol{D}.$
\item
If $\ell\equiv 0,4,$ then $D$ is $\K$-linear with no further conditions imposed.
\end{itemize}
Note that $\ell$ may be different from the dimension $n.$
\end{dfn}

This terminology is motivated by the properties of the real Dirac operator on an $n$-dimensional spin manifold. Appendix~\ref{bioB} reviews real Dirac operators, and Table~\ref{bioBtab2} shows that the (untwisted) real Dirac operator is $\ell$-adapted for $\ell=n.$ 

We call zeroth order differential operator \emph{potentials}. Let $\cV_\ell\subset\Ga^\iy(E^*_0\ot_\K E_1)$ be the Fréchet subspace of $\ell$-adapted potentials.

An \emph{isomorphism} of $\ell$-adapted differential operators $D\colon\Ga^\iy(E_0)\to\Ga^\iy(E_1)$ and $D'\colon\Ga^\iy(E_0')\to\Ga^\iy(E_1')$ on $X$ is a pair of metric $\K$-linear vector bundle isomorphisms $\Phi_0\colon E_0\to E_0',$ $\Phi_1\colon E_1\to E_1'$ satisfying $D'\circ\Phi_0=\Phi_1\circ D$ and $\Phi_0=\Phi_1$ if $\ell\equiv 1,3,7,$ $\Phi_0=\ol{\Phi_1}$ if $\ell\equiv 2,6,$ and $\Phi_0=\Phi_1^\diamond$ if $\ell\equiv 5.$ The pullback of an $\ell$-adapted differential operator $D$ along an open embedding is again $\ell$-adapted.


\begin{dfn}
\label{bio2def2}
Let $D$ be an $\ell$-adapted elliptic differential operator on a compact manifold $X$ without boundary. Then $D$ is a Fredholm operator, see Lawson--Michelsohn~\cite{LaMi}*{Ch.~III, \S 5}, and the \emph{index} of $D$ in $\Ga_\ell$ is defined as
\e
 \ind_\ell(D)
 =
 \begin{cases}
 \dim_\K\Ker(D) - \dim_\K\Coker(D) & \text{if $\ell\equiv 0,4,$}\\
 \dim_\K\Ker(D)\bmod 2 & \text{if $\ell\equiv 1,2,$}\\
 0 & \text{if $\ell\equiv 3,5,6,7.$}
 \end{cases}
 \label{bio2eq1}
\e
Note here that $D$ is skew-adjoint if $\ell\equiv 1,2,$ so \eqref{bio2eq1} is well-behaved.
\end{dfn}

For the real Dirac operator on a spin $n$-manifold, this definition agrees with the index in real K-theory (the $\hat A$-genus), see \cite{LaMi}*{Ch.~III, \S 16}.

\subsection{Orientation bundles}
\label{bio22}

The orientation bundle of a family $\cD_X$ of $\ell$-adapted elliptic differential operators will be a $\Ga_\ell$-graded principal $\Ga_{\ell+1}$-bundle, which `categorifies' the first Stiefel--Whitney class of the families index.

\begin{dfn}
\label{bio2def3}

Let $G,$ $H$ be discrete abelian groups. A \emph{$G$-graded principal $H$-bundle} over a topological space $T$ is a pair $(g,P)$ of a locally constant map $g\colon T\to G$ and a principal $H$-bundle $P\to T.$ These form the objects of a category $(G/\!\!/H)^T$ in which there are morphisms $\Phi\colon(g_0,P_0)\to (g_1,P_1)$ only if $g_0=g_1$ and are then defined to be all isomorphisms $\Phi\colon P_0\to P_1$ of principal bundles.

The \emph{monoidal structure} is $(g,P)\ot(g',P')=(g+g',(P\t P')/H),$ where $H$ acts anti-diagonally on $P\t P'.$ The \emph{unit object} is $\boo_{(G/\!\!/H)^T}=(0_G,H\t T).$ There are \emph{symmetry isomorphisms} $(g+g',[p,p'])\mapsto(g'+g,[(-1)^{g+g'}p',p]).$ Hence $(G/\!\!/H)^T$ is a symmetric monoidal category in the sense of MacLane~\cite{MacL}*{Ch.~XI.1}. Every morphism is invertible and every object $(g,P)$ has a \emph{dual object} $(-g,P^*),$ where $P^*=\Hom_H(P,H),$ so there is an isomorphism $(g,P)\ot(-g,P^*)\to\boo,$ $(t,\varphi)\mapsto\an{t,\varphi}.$ In other words, $(G/\!\!/H)^T$ is a Picard groupoid as in Sinh~\cite{Sinh}.

If $T$ is a point, we speak of a $G$-graded $H$-torsor and simply write $G/\!\!/H.$
\end{dfn}

For a finite-dimensional real vector space $V$ let $\det_\R(V)=\La^{\dim V}_\R(V)$ and $O(V)=(\det(V)\setminus\{0\})/\R_{>0}$ be the $\Z_2$-torsor of orientations $\{\pm\om\}$ on $V.$ Applied fiberwise to a real vector bundle $V\to T,$ one obtains the principal $\Z_2$-bundle $O(V)\to T$ of fiberwise orientations.

\begin{dfn}
\label{bio2def4}

Let $X$ be a compact manifold without boundary, $T$ a paracompact Hausdorff space, and $E_0, E_1 \to X\t T$ vector bundles that are smooth in the $X$-direction. Let $\cD_X=\{D_X(t)\}$ be a continuous $T$-family of $\ell$-adapted elliptic differential operators, where $D_X(t)\colon\Ga^\iy\bigl(E_0\vert_{X\t\{t\}}\bigr)\to\Ga^\iy\bigl(E_1\vert_{X\t\{t\}}\bigr)$ for $t\in T.$
\begin{itemize}
\item
If $\ell\equiv 0,$ then $\cD_X$ determines a continuous $T$-family of $\R$-linear Fredholm operators, so by \S\ref{bioA2} a real \emph{determinant line bundle} $\DET_\R\cD_X\to T$ with fibers $\DET_\R\cD_X|_t\cong\det_\R\Ker D_X(t)\ot\det_\R(\Coker D_X(t))^*$ and a principal $\Z_2$-bundle $O(\DET_\R \cD_X)\to T.$ Thus, for $\ell\equiv0$ we study orientations on the `virtual' vector spaces $\Ker D_X(t)-\Coker D_X(t).$
\item
If $\ell\equiv1,$ then $\cD_X$ determines a continuous $T$-family of $\R$-linear skew-adjoint Fredholm operators, so by \S\ref{bioA3} a \emph{Pfaffian line bundle} $\PF\cD_X\to T$ with fibers $\PF \cD_X|_t\cong\det_\R\Ker D_X(t)$ and a principal $\Z_2$-bundle $O(\PF_\R \cD_X)\to T.$ Thus, for $\ell\equiv1$ we study orientations on the vector spaces $\Ker D_X(t).$
\item
If $\ell\equiv3,7,$ then $\cD_X$ determines a continuous $T$-family of $\K$-linear self-adjoint Fredholm operators with $\K=\H,\R$ as in Table~\ref{bio2tab1}, so by \S\ref{bioA4} there is a \emph{spectral principal $\Z$-bundle} $\SP_\K\cD_X\to T$ with fibers $\SP_\K\cD_X|_t$ the set of all enumerations $\cdots\le\la_{-1}\le\la_0\le\la_1\le\cdots$ of the spectrum of $D_X(t),$ where eigenvalues are repeated with multiplicity over $\K.$ The principal action by $\Z$ translates the labels.
(By \cite{LaMi}*{Ch.~III, \S 5} the operator $D_X(t)$ has a countable spectrum consisting entirely of eigenvalues.)
\end{itemize}

The \emph{orientation bundle} $\rO_\ell(\cD_X)$ is the $\Ga_\ell$-graded principal $\Ga_{\ell+1}$-bundle with grading $t\mapsto\ind_\ell D_X(t)$ and the underlying principal $\Ga_{\ell+1}$-bundle $O(\DET_\R\cD_X),$ $O(\PF_\R\cD_X),$ $\SP_\K\cD_X$ for $\ell\equiv 0,1,3,7$ and trivial otherwise (as $\Ga_{\ell+1}=0$).
\end{dfn}


In each case $-D_X(t)^*$ is again a $T$-family of $\ell$-adapted elliptic differential operators and from Appendix~\ref{bioA} there are isomorphisms $\rO_\ell(-\cD_X^*)\cong\rO_\ell(\cD_X)^*.$

\subsection{Elliptic operator bordisms}
\label{bio23}

The next definition is motivated by the relationship between the real Dirac operator on a spin manifold $X$ and on $X\t\R,$ see Table~\ref{bioBtab2} in Appendix~\ref{bioB3}. This will model the behaviour of differential operators over a collar neighborhood $X\t[0,\varepsilon)\subset X\t\R$ of a manifold with boundary.

\begin{dfn}
\label{bio2def5}

Let $D\colon\Ga^\iy(E_0)\to\Ga^\iy(E_1)$ be an $\ell$-adapted differential operator on $X.$  Let $t$ be the coordinate on $\R.$ The \emph{cylinder} $\Cyl(D)\colon\Ga^\iy(F_0)\to\Ga^\iy(F_1)$ and the vector bundles $F_0, F_1\to X\t\R$ are defined as follows.

\begin{itemize}
\item
If $\ell\equiv0,$ let $F_0=F_1=(E_0\op E_1)\t\R$ and
\e
\label{bio2eq2}
\Cyl(D)=\begin{pmatrix}
  	\frac{\6}{\6 t} & D^*\\
  	-D & -\frac{\6}{\6 t}
  \end{pmatrix}.
\e

\item
If $\ell\equiv1,$ then $E_0=E_1,$ $D^*=-D.$ Let $F_0=(E_0\ot_\R \C)\t\R,$ $F_1=\ol{F_0},$ and
\e
\label{bio2eq3}
 \Cyl(D)=\si\left(\frac{\6}{\6 t}+iD_\C\right),
\e
where $\si\colon F_0\to F_1,$ $\si(e\ot z)=e\ot\ol{z},$ and $D_\C$ is the complexification of $D.$

\item
If $\ell\equiv2,$ then $E_0=\ol{E_1}$ and $D^*=-\ol D.$ Let $F_0=F_1=(\H\ot_\C E_0)\t\R.$ Using the decomposition $\H\ot_\C E_0\cong E_0\op\ol{E_0},$ define
\e
\label{bio2eq4}
 \Cyl(D)=
 \begin{pmatrix}
-i\frac{\6}{\6 t}& -\ol{D}\\
D & i\frac{\6}{\6 t}
\end{pmatrix}.
\e
\item
If $\ell\equiv3,7,$ then $E_0=E_1$ and $D^*=D.$ Let $F_0=F_1=E_0\t\R$ and
\e
\label{bio2eq5}
 \Cyl(D)=\frac{\6}{\6 t}+D.
\e
\item
If $\ell\equiv4,$ then $F_0=(E_0\op E_1^\diamond)\t\R,$ $F_1=F_0^\diamond=(E_0^\diamond\op E_1)\t\R,$ and
\e
\label{bio2eq6}
 \Cyl(D)=
 \begin{pmatrix}
 i\frac{\6}{\6 t}& {D^*}^\diamond\\
 D & -i\frac{\6}{\6 t}
 \end{pmatrix}.
\e
\item
If $\ell\equiv5,$ then $E_0=E_1^\diamond$ and $D^*=D^\diamond.$ Let $F_0=E_0\t\R,$ $F_1=\ol{F_0},$ and
\e
\label{bio2eq7}
 \Cyl(D)=j\Bigl(\frac{\6}{\6 t}+D\Bigr).
\e
As $\6/\6 t$ commutes and $D$ anti-commutes with $j,$ $\Cyl(D)$ is only $\C$-linear.
\item
If $\ell\equiv6,$ then $E_0=\ol{E_1}$ and $D^*=\ol{D}.$ Let $F_0=F_1=E_0\t\R$ and
\e
\label{bio2eq8}
 \Cyl(D)=i\frac{\6}{\6 t}+D.
\e
As $i\6/\6 t$ commutes and $D$ anti-commutes with $i,$ $\Cyl(D)$ is only $\R$-linear.
\end{itemize}

It is easy to check in each case that $\Cyl(D)$ is an $(\ell+1)$-adapted differential operator. By iterating the construction, define $\Cyl^k(D)$ on $X\t\R^k$ for all $k\in\N.$
\end{dfn}

\begin{dfn}
\label{bio2def6}

Let $X_0,$ $X_1$ be compact $n$-manifolds. Recall that a \emph{bordism} from $X_0$ to $X_1$ is a compact $(n+1)$-manifold $Y$ with boundary $\6 Y\cong X_0\amalg X_1$ and a (chosen) collar neighborhood $\io_Y\colon(X_0\t[0,\varepsilon_0))\amalg(X_1\t(\varepsilon_1,1])\to Y.$

Let $D_{X_0},$ $D_{X_1}$ be $\ell$-adapted elliptic differential operators on $X_0,$ $X_1.$ An \emph{elliptic operator bordism} from $(X_0,D_{X_0})$ to $(X_1,D_{X_1})$ is a bordism $Y$ from $X_0$ to $X_1,$ an $(\ell+1)$-adapted elliptic differential operator $D_Y$ on $Y,$ and an isomorphism
\e
\label{bio2eq9}
 \io_Y^*(D_Y)\cong\Cyl(D_{X_0})|_{X_0\t[0,\varepsilon_0)}\amalg\Cyl(D_{X_1})|_{X_1\t(\varepsilon_1,1]}.
\e

Two bordisms $Y_0,$ $Y_1$ are \emph{equivalent} if there is a compact $(n+2)$-manifold $Z$ with corners and $\6 Z\cong (X_0\t[0,1])\amalg(X_1\t[0,1])\amalg Y_0\amalg Y_1$ and a collar
\[
 (X_0\t[0,\varepsilon_0)\t[0,1])
 \amalg
 (X_1\t(\varepsilon_1,1]\t[0,1])
 \amalg
 (Y_0\t[0,\de_0))
 \amalg
 (Y_1\t(\de_1,1])
 \overset{\io_Z}{\longra}
 Z
\]
such that $\io_Z(x_i,s,t)=\io_Z(\io_{Y_j}(x_i,s),t)$ for all $i,j\in\{0,1\},$ $x_i\in X_i,$ $(-1)^i(s-i)\in[0,\varepsilon_i),$ and $(-1)^j(t-j)\in[0,\de_j),$ see Figure~\ref{collar-nghd}. If $(Y_0,D_{Y_0}),$ $(Y_1,D_{Y_1})$ are elliptic operator bordisms, there should also exist an $(\ell+2)$-adapted elliptic differential operator $D_Z$ on $Z$ and an isomorphism $\io_Z^*(D_Z)\cong\Cyl^2(D_{X_0})\amalg\Cyl^2(D_{X_1})\amalg\Cyl(D_{Y_0})\amalg\Cyl(D_{Y_1})$ which identifies $(\io_{Y_0}\t\id)^*\Cyl(D_{Y_0})\cong\Cyl^2(D_{X_0})$ and $(\io_{Y_1}\t\id)^*\Cyl(D_{Y_1})\cong\Cyl^2(D_{X_1})$ under \eqref{bio2eq9}.

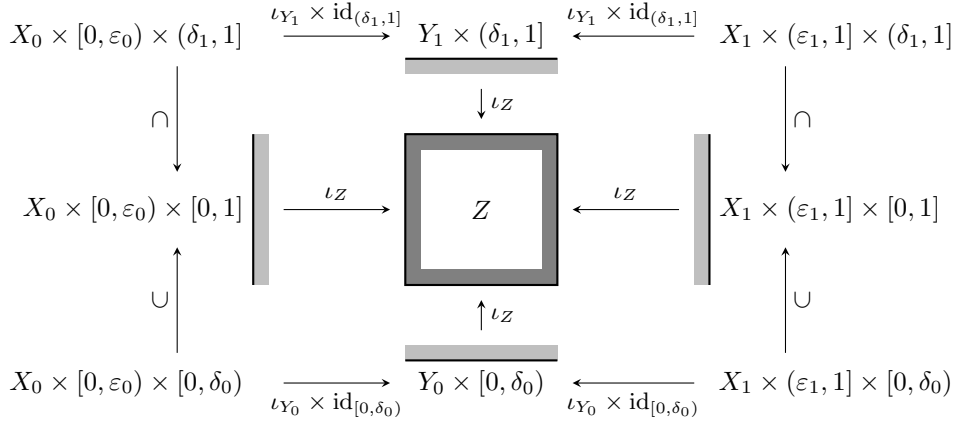
\begin{figure}[htb]
	\centering
	\begin{tikzpicture}[scale=2]

		\draw[fill=gray,opacity=0.5,draw=none] (0,0) rectangle (0.1,1);
		\draw[fill=gray,opacity=0.5,draw=none] (0.9,0) rectangle (1,1);
		\draw[fill=gray,opacity=0.5,draw=none] (0,0) rectangle (1,0.1);
		\draw[fill=gray,opacity=0.5,draw=none] (0,0.9) rectangle (1,1);
		\draw[thick] (0,0) rectangle node{$Z$} (1,1);
		
		\draw[fill=lightgray,draw=none] (-1,0) rectangle (-0.9,1);
		\draw[thick] (-1,0) -- node[left]{$X_0\t[0,\varepsilon_0)\t[0,1]$} (-1,1);
		\draw[-stealth] (-0.8,0.5) -- node[above]{\small $\io_Z$} (-0.1,0.5);

		\draw[fill=lightgray,draw=none] (2,0) rectangle (1.9,1);
		\draw[thick] (2,0) -- node[right]{$X_1\t(\varepsilon_1,1]\t[0,1]$} (2,1);
		\draw[-stealth] (1.8,0.5) -- node[above]{\small $\io_Z$} (1.1,0.5);

		\draw[fill=lightgray,draw=none] (0,-0.5) rectangle (1,-0.4);
		\draw[thick] (0,-0.5) -- node[below]{$Y_0\t[0,\de_0)$} (1,-0.5);
		\draw[-stealth] (0.5,-0.3) -- node[right]{\small $\io_Z$} (0.5,-0.1);

		\draw[fill=lightgray,draw=none] (0,1.4) rectangle (1,1.5);
		\draw[thick] (0,1.5) -- node[above]{$Y_1\t(\de_1,1]$} (1,1.5);
		\draw[-stealth] (0.5,1.3) -- node[right]{\small $\io_Z$} (0.5,1.1);

		\draw[thick] (-1,-0.65) node[left]{$X_0\t[0,\varepsilon_0)\t[0,\de_0)$};
		\draw[-stealth] (-0.8,-0.65) -- node[below]{\small $\io_{Y_0}\t\id_{[0,\de_0)}$} (-0.1,-0.65);
		\draw[-stealth] (-1.5,-0.45) -- node[above,rotate=90]{$\subset$} (-1.5,0.25);
		
		\draw (2,-0.65) node[right] {$X_1\t(\varepsilon_1,1]\t[0,\de_0)$};
		\draw[-stealth] (1.9,-0.65) -- node[below]{\small $\io_{Y_0}\t\id_{[0,\de_0)}$} (1.1,-0.65);
		\draw[-stealth] (2.5,-0.45) -- node[below,rotate=90]{$\subset$} (2.5,0.25);

		\draw[thick] (-1,1.65) node[left]{$X_0\t[0,\varepsilon_0)\t(\de_1,1]$};
		\draw[-stealth] (-0.8,1.65) -- node[above]{\small $\io_{Y_1}\t\id_{(\de_1,1]}$} (-0.1,1.65);
		\draw[-stealth] (-1.5,1.45) -- node[below,rotate=-90]{$\subset$} (-1.5,0.75);

		\draw[thick] (2,1.65) node[right]{$X_1\t(\varepsilon_1,1]\t(\de_1,1]$};
		\draw[-stealth] (1.9,1.65) -- node[above]{\small $\io_{Y_1}\t\id_{(\de_1,1]}$} (1.1,1.65);
		\draw[-stealth] (2.5,1.45) -- node[above,rotate=-90]{$\subset$} (2.5,0.75);

	\end{tikzpicture}
	\caption{Equivalent bordisms $Y_0,$ $Y_1$ and the collar neighborhood of the manifold with corners $Z$}
	\label{collar-nghd}
\end{figure}

The \emph{composition} of $(Y,D_Y)$ with another elliptic operator bordisms $(\tilde Y,D_{\tilde Y})$ from $(X_1,D_{X_1})$ to $(X_2,D_{X_2})$ is given by the glued manifold $Y\cup_{X_1}\tilde{Y},$
using the collars to define the smooth structure over $(-\varepsilon_1,\tilde\varepsilon_1)\t X_1\to Y\cup_{X_1}\tilde{Y},$
and the $(\ell+1)$-adapted elliptic differential operators obtained by gluing $D_Y$ to $D_{\tilde Y},$ which is possible because of the cylindrical forms \eqref{bio2eq2}--\eqref{bio2eq8} over the collars.

The $\ell$-adapted elliptic differential operator $(X,D_X)$ on compact $n$-manifolds $X$ are the objects of a \emph{bordism category} $\Bord_n^{\Ell_\ell}$ in which the morphisms $(X_0,D_{X_0})\to(X_1,D_{X_1})$ are equivalence classes $[Y,D_Y]$ of elliptic operator bordisms. The \emph{identity morphism} at $(X,D_X)$ is $[X\t[0,1],\Cyl(D_X)].$

The disjoint union $(X,D_X)\ot(X',D_{X'})=(X\amalg X',D_X\amalg D_{X'})$ defines a \emph{monoidal structure} on $\Bord_n^{\Ell_\ell}$ with the empty manifold as the \emph{unit object} $\boo_{\Bord_n^{\Ell_\ell}}.$ There are \emph{symmetry isomorphisms} $[(X\amalg X')\t[0,1],\Cyl(D_X\amalg D_{X'})]$ where the collar exchanges the roles of $X$ and $X'.$ With these operations, $\Bord_n^{\Ell_\ell}$ becomes a symmetric monoidal category.
\end{dfn}

Let $\rho\colon X\t\R\to X\t\R,$ $(x,t)\mapsto(x,-t).$ Pulling back $\Cyl(D_X)$ along $\rho$ has the effect of reversing the sign of $\6/\6 t$ in \eqref{bio2eq2}--\eqref{bio2eq8} and passing to the formal adjoint $-D_X^*$ exchanges the roles of $E_0$ and $E_1.$ Observe that the pullback $\rho^*\Cyl(D_X)$ is isomorphic to $\Cyl(-D_X^*)$ as an $(\ell+1)$-adapted differential operator for all $\ell.$
Hence an elliptic operator bordism $(Y,D_Y)$ can equivalently be viewed as a morphism
\e
\begin{aligned}
 (X_0,D_{X_0})&\longra (X_1,D_{X_1}),\\
 \boo&\longra(X_0,-D_{X_0}^*)\ot(X_1,D_{X_1}),\\
 (X_0,D_{X_0})\ot(X_1,-D_{X_1}^*)&\longra\boo,\\
 (X_1,-D_{X_1}^*)&\longra (X_0,-D_{X_0}^*).
\end{aligned}
\label{bio2eq10}
\e

Applied to $\Cyl(D_X)$ we see that each object $(X,D_X)$ has a \emph{dual object} $(X,-D_X^*)$ in the sense that $\Cyl(D_X)$ defines an isomorphism $(X,D_X)\ot(X,-D_X^*)\to\boo.$ Hence $\Bord_n^{\Ell_\ell}$ is a Picard groupoid.

\section{Main results}
\label{bio3}

\subsection{Statements of main results}
\label{bio31}

We use the notation explained in \S\ref{bio2}. In particular, recall that every $\ell$-adapted first order elliptic differential operators $D_X$ on a compact manifold without boundary has an orientation $\Ga_{\ell+1}$-torsor $\rO_\ell(D_X).$

\begin{thm}
\label{bio3thm1}
Every\/ $(\ell+1)$-adapted elliptic operator bordism
\e
 [Y,D_Y]\colon (X_0,D_{X_0})\longra(X_1,D_{X_1})
\label{bio3eq1}
\e
of\/ $\ell$-adapted elliptic differential operators on compact\/ $n$-manifolds\/ $X_0,$ $X_1$ without boundary induces an isomorphism of graded orientation torsors,
\e
\label{bio3eq2}
 \rO_\ell[D_Y]\colon \rO_\ell(D_{X_0})\longra\rO_\ell(D_{X_1}).
\e
The isomorphism \eqref{bio3eq2} depends only on the equivalence class\/ $[Y,D_Y].$ In particular, on the level of gradings we obtain the bordism invariance of the classical index,
\[
 \ind_\ell(D_{X_0})=\ind_\ell(D_{X_1}).
\]
Moreover, the following additional properties hold.
\begin{enumerate}
\item
\textup(Empty boundary\textup)
If\/ $X_0=X_1=\es,$ the elliptic differential operator\/ $D_Y$ on the compact manifold\/ $Y$ without boundary has an index in\/ $\Ga_{\ell+1}$ as in \eqref{bio2eq1}. Under\/ $\Aut(\boo_{\KOtor})\cong\Ga_{\ell+1}$ we have\/ $\rO_\ell[D_Y]=\ind_{\ell+1}(D_Y).$
\item
\textup(Functoriality and disjoint unions\textup)
The construction \eqref{bio3eq2} is functorial for the composition of elliptic operator bordisms and compatible with disjoint un\-ions, so determines a symmetric monoidal functor
\[
 \rO_\ell\colon\Bord_n^{\Ell_\ell}\longra\KOtor
\]
from the elliptic bordism category to\/ $\Ga_\ell$-graded\/ $\Ga_{\ell+1}$-torsors.
\item
\textup(Continuity\textup)
Let\/ $T$ be a paracompact Hausdorff topological space. Let\/ $\cD_{X_0}=\{D_{X_0}(t)\}$ and\/ $\cD_{X_1}=\{D_{X_1}(t)\}$ be\/ $T$-families of\/ $\ell$-adapted elliptic differential operators on\/ $X_0$ and\/ $X_1,$ so the graded orientation torsors naturally assemble into principal\/ $\Ga_{\ell+1}$-bundles\/ $\rO_\ell(\cD_{X_0})=\coprod_{t\in T}\rO_\ell(D_{X_0}(t))$ and\/ $\rO_\ell(\cD_{X_1})=\coprod_{t\in T}\rO_\ell(D_{X_1}(t))$ over\/ $T.$ If\/ $\cD_Y=\{D_Y(t)\}$ is a\/ $T$-family of\/ $(\ell+1)$-adapted elliptic operator bordisms on\/ $Y,$ then \eqref{bio3eq2} is a continuous isomorphism\/ $\rO_\ell[\cD_Y]=\coprod_{t\in T}\rO_\ell[D_Y(t)]$ from\/ $\rO_\ell(\cD_{X_0})$ to\/ $\rO_\ell(\cD_{X_1}).$
\end{enumerate}
\end{thm}

Theorem~\ref{bio3thm1} is a direct consequence of the following more technical result. To explain this note that by \eqref{bio2eq10} we can regard \eqref{bio3eq1} as a morphism $\boo\to(X_0\amalg X_1, -D_{X_0}^*\amalg D_{X_1}).$ If suffices then to define an isomorphism $\Ga_{\ell+1}\to\rO_\ell(-D_{X_0}^*\amalg D_{X_1})$ of torsors since $\rO_\ell(-D_{X_0}^*\amalg D_{X_1})\cong \rO_\ell(-D_{X_0}^*)\ot\rO_\ell(D_{X_1})\cong\rO_\ell(D_{X_0})^*\ot\rO_\ell(\cD_{X_1})=\Hom(\rO_\ell(D_{X_0}),\rO_\ell(D_{X_1}))$; (a) is then a consequence of Theorem~\ref{bio3thm2}(e) and (b) follows from parts (a) and (c) of Theorem~\ref{bio3thm2}.

\begin{thm}
\label{bio3thm2}
Let\/ $Y$ be a compact manifold with boundary\/ $\6 Y\cong X.$ Let\/ $\cD_X=\{D_X(t)\}$ be a\/ $T$-family of\/ $\ell$-adapted elliptic differential operators on\/ $X$. If\/ $\cD_Y=\{D_Y(t)\}$ is a\/ $T$-family of\/ $(\ell+1)$-adapted elliptic differential operators on\/ $Y$ that is isomorphic to the cylinder\/ $\Cyl \cD_X$ over a collar, then
\e
\label{bio3eq3}
 \ind_\ell D_X(t)=0
\e
for all\/ $t\in T.$ Moreover,\/ $\cD_Y$ induces an isomorphism
\e
 \tau_\ell[\cD_Y]\colon\rO_\ell(\cD_X)\overset{\cong}{\longra}\Ga_{\ell+1}\t T.
\label{bio3eq4}
\e
\textup(As explained in \textup{\S\ref{bio4}}, elements of\/ $\rO_\ell(\cD_X)$ determine elliptic boundary conditions for\/ $\cD_Y$ and \eqref{bio3eq4} takes the index of the\/ $(\ell+1)$-adapted elliptic differential operator\/ $D_Y(t)$ for all\/ $t\in T$ with respect to that boundary condition.\textup)

The isomorphism \eqref{bio3eq4} has the following properties.

\begin{enumerate}
\item\textup(Gluing\textup)
\label{bio3thm2a}
Suppose additionally that\/ $Y$ contains a cylinder\/ $X_0\t(-\varepsilon,\varepsilon)$ over which\/ $D_Y(t)$ is isomorphic to\/ $\Cyl D_{X_0}(t)$ for a\/ $T$-family of\/ $\ell$-adapted elliptic differential operator\/ $\cD_{X_0}=\{D_{X_0}(t)\}$ on\/ $X_0.$ Let\/ $Y'$ be the manifold obtained by cutting\/ $Y$ along\/ $X_0,$ so\/ $\6 Y'$ consists of\/ $X=\6 Y$ and two copies\/ $X_\pm$ of\/ $X_0.$  The pullback differential operator\/ $D_{Y'}(t)$ on\/ $Y'$ then has the form\/ $\Cyl D_X(t)\amalg \Cyl D_{X_+}(t)\amalg \Cyl D_{X_-}(t)$ over a collar, where\/ $D_{X_+}(t)=D_{X_0}(t),$ $D_{X_-}(t)=-D_{X_0}(t)^*.$ There is a commutative diagram
\begin{equation}
\begin{tikzcd}[column sep=small]
	\rO_\ell(\cD_X)\ot\rO_\ell(\cD_{X_+})\ot\rO_\ell(\cD_{X_-})\arrow[r,"\cong"]
&
	 \rO_\ell(\cD_X)\ot\rO_\ell(\cD_{X_0})\ot\rO_\ell(\cD_{X_0})^*\arrow[d,"{\tau_\ell[\cD_Y]\ot\an{\,,\,}}"]
\\
\rO_\ell(\cD_{X\amalg X_+\amalg X_-})
\arrow[u,"\cong"]
\arrow[r,"{\tau_\ell[\cD_{Y'}]}"]
&
\Ga_{\ell+1}\t T.
\end{tikzcd}
\label{bio3eq5}
\end{equation}
\item
\textup(Bordism invariance\textup)
The isomorphism \eqref{bio3eq4} only depends on the bordism class of\/ $(Y,\cD_Y)$ relative to\/ $(X,\cD_X).$ That is, let\/ $Y,$ $Y'$ be manifolds with the same boundary\/ $\6 Y\cong X\cong\6 Y'$ and let\/ $Z$ be a manifold with corners such that\/ 
$\6 Z\cong Y\amalg Y'\amalg (X\t[0,1]),$ identifying\/ $\6 Y$ with\/ $X\t\{0\}$ and\/ $\6 Y'$ with\/ $X\t\{1\}$ in the obvious way. Let\/ $\cD_Y=\{D_Y(t)\},$ $\cD_{Y'}=\{D_{Y'}(t)\}$ be\/ $T$-families of\/ $(\ell+1)$-adapted elliptic differential operators on\/ $Y,$ $Y'$ isomorphic over a collar to\/ $\Cyl D_X(t).$ If\/ $D_Z(t)$ is an\/ $(\ell+2)$-adapted elliptic differential operator on\/ $Z$ isomorphic over a collar of\/ $\6 Z$ to\/ $\Cyl D_Y(t)\amalg \Cyl D_{Y'}(t)\amalg \Cyl^2 D_X(t),$ then\/ $\tau_\ell[\cD_{Y}]=\tau_\ell[\cD_{Y'}].$
\item
\label{bio3thm2c}
\textup(Disjoint union\textup)
Let\/ $Y,$ $Y'$ be compact manifolds with\/ $X\cong\6 Y,$ $X'\cong\6 Y'.$ For the disjoint union of\/ $T$-families of adapted elliptic differential operators on\/ $X\amalg X'$ and on\/ $Y\amalg Y'$ there is a commutative diagram
\[
\begin{tikzcd}[column sep=20ex]
	\rO_\ell(\cD_X)\ot\rO_\ell(\cD_{X'})\arrow[d]\arrow[r,"{\tau_\ell[\cD_Y]\ot\tau_\ell[\cD_{Y'}]}"]&	(\Ga_{\ell+1}\t T)\ot (\Ga_{\ell+1}\t T)\arrow[d,"\cong"]\\
	\rO_\ell(\cD_{X\amalg X'})\arrow[r,"{\tau_\ell[\cD_Y\amalg\cD_{Y'}]}"] & \Ga_{\ell+1}\t T.
\end{tikzcd}
\]
\item
\label{bio3thm2d}
\textup(Direct sums\textup)
Let\/ $\cD_Y,$ $\tilde \cD_Y$ be two\/ $T$-families of\/ $(\ell+1)$-adapted elliptic differential operators on a compact manifold\/ $Y.$ Suppose that over a collar\/ $D_Y(t)$ is isomorphic to\/ $\Cyl D_X(t)$ and\/ $\tilde D_Y(t)$ to\/ $\Cyl \tilde D_X(t).$ For the direct sum of differential operators there is a commutative diagram
\[
 \begin{tikzcd}[column sep=15ex]
 	\rO_\ell(\cD_X)\ot\rO_\ell(\tilde\cD_X)\arrow[d,"\cong"]\arrow[r,"{\tau_\ell[\cD_Y]\ot\tau_\ell[\tilde\cD_Y]}"] & (\Ga_{\ell+1}\t T)\ot(\Ga_{\ell+1}\t T)\arrow[d,"\cdot"]\\\rO_\ell(\cD_X\op\tilde\cD_X)\arrow[r,"{\tau_\ell[\cD_Y\op\tilde\cD_Y]}"] & \Ga_{\ell+1}\t T.
 \end{tikzcd}
\]
\item
\label{bio3thm2e}
\textup(Empty boundary\textup)
If\/ $X=\emptyset,$ then\/ $\rO_\ell(\emptyset)=\Ga_{\ell+1}\t T$ canonically and
\[
 \tau_\ell[\cD_Y]\colon(\ga,t)\longmapsto(\ga+\ind_{\ell+1} D_Y(t),t).
\]
\end{enumerate}
\end{thm}


The proof of Theorem~\ref{bio3thm2} will be given in \S\ref{bio43}. Since part (b) is a consequence of (a), (c), and (e), we will prove it here.

\begin{proof}[Proof of Theorem~\textup{\ref{bio3thm2}(b)}.]

From $Y,$ $Y'$ we can construct compact manifolds \emph{without boundary} $Y\amalg_X Y',$ $Y'\amalg_X Y',$ and a bordism $Z\amalg_{X\t[0,1]}(Y'\t[0,1])$ without corners between them. Define $T$-families of $(\ell+1)$-adapted differential operators $\cD_{Y\amalg_X Y'}=\cD_Y\amalg_X \cD_{Y'}$ and $\cD_{Y'\amalg_X Y'}=\cD_{Y'}\amalg_X \cD_{Y'}.$ Applying (a) twice, to $Y\amalg_X Y'$ and to $Y'\amalg_X Y'$ cut along $X_0=X,$ gives a pair of commutative diagrams (left and right)
\[
\begin{tikzcd}[column sep=huge]
	\rO_\ell(\cD_{X_+})\ot\rO_\ell(\cD_{X_-})\arrow[r,"\cong"]\arrow[d,shift left=0.5ex,"{\tau_\ell[\cD_Y]\ot\tau_\ell[\cD_{Y'}]}"]
	\arrow[d,shift right=0.5ex,"{\tau_\ell[\cD_{Y'}]\ot\tau_\ell[\cD_{Y'}]}"'] & \rO_\ell(\cD_X)\ot\rO_\ell(\cD_X)^*
	\arrow[d,shift left=0.5ex,"{\ind_{\ell+1}(\cD_{Y\amalg_X Y'})\an{,}}"]
	\arrow[d,shift right=0.5ex,"{\ind_{\ell+1}(\cD_{Y'\amalg_X Y'})\an{,}}"']
	\\
	(\Ga_{\ell+1}\t T)\ot(\Ga_{\ell+1}\t T)\arrow[r,"\cong"] & \Ga_{\ell+1}\t T,
\end{tikzcd}
\]
where we have used (c) and (e) to rewrite the vertical maps; the horizontal maps are the same in both cases. Hence $\tau_\ell[\cD_Y]=\tau_\ell[\cD_{Y'}]$ follows from $\ind_{\ell+1}(\cD_{Y\amalg_X Y'})=\ind_{\ell+1}(\cD_{Y'\amalg_X Y'})$ which holds by \eqref{bio3eq3} applied to $\cD_{Y\amalg_X Y'}\amalg-\cD_{Y'\amalg_X Y'}^*$ and the $T$-family $\cD_{Z\amalg_{X\t[0,1]}(Y'\t[0,1])}=\cD_Z\amalg_{X\t[0,1]}\Cyl \cD_{Y'}$ on $Z.$
\end{proof}

While $\rO_\ell(D_X)$ and $\rO_\ell[D_Y]$ in Theorem~\ref{bio3thm1} appear to depend on the differential operators, they actually only depends on their principal symbols. This is a simple consequence of part (c), as we now explain.

\begin{dfn}
\label{bio3def1}
Let $\si_X\colon E_0\ot T^*X\to E_1$ be the principal symbol of an $\ell$-adapted elliptic differential operator on $X$; we call $\si_X$ an \emph{$\ell$-adapted elliptic symbol}. Let $T_{\si_X}^\ell$ be the set of $\ell$-adapted elliptic differential operators $D_X$ with $\si(D_X)=\si_X.$ Observe that there is a tautological $T_{\si_X}^\ell$-family $\cD_X^{\si_X}$ of $\ell$-adapted operators. Moreover, $T_{\si_X}^\ell$ is an affine space over the space of $\ell$-adapted potentials $\cV_\ell$ as in Definition~\ref{bio2def1} so, in particular, $T_{\si_X}^\ell$ has a natural Fr\'echet space topology. Since $T_{\si_X}^\ell$ is a contractible topological space, the $\Ga_\ell$-graded principal $\Ga_{\ell+1}$-bundle $\rO_\ell(\cD_X^{\si_X})\to T_{\si_X}^\ell$ is trivial and hence the grading is a constant function and the global sections $\rO_\ell(\si_X)=\Ga^\iy(\rO_\ell(\cD_X^{\si_X}))\in\Ga_\ell/\!\!/\Ga_{\ell+1}$ form a nonempty $\Ga_{\ell+1}$-torsor, called the \emph{orientation torsor} of the $\ell$-adapted elliptic symbol $\si_X.$ For each $\ell$-adapted elliptic differential operator $D_X$ with $\si(D_X)=\si_X$ the evaluation of the section at $D_X$ defines an isomorphism
\e
 \rO_\ell(\si_X)\overset{\smash\cong}{\longra}\rO_\ell(D_X).
 \label{bio3eq6}
\e

Similarly, an $(\ell+1)$-adapted elliptic symbol $\si_Y$ on a bordism $Y$ determines a $T_{\si_Y}^{\ell+1}$-family $\cD_Y^{\si_Y}$ of $(\ell+1)$-adapted operators on $Y.$ Suppose that
\[
 \io_Y^*(\si_Y)\cong\Cyl(\si_{X_0})|_{X_0\t[0,\varepsilon_0)}\amalg\Cyl(\si_{X_1})|_{X_1\t(\varepsilon_1,1]},
\]
where $\Cyl(\si)$ is defined as in \eqref{bio2eq2}--\eqref{bio2eq8} by replacing $D$ by $\si.$ The `restriction' $r\colon T_{\si_Y}^{\ell+1}\to T_{\si_{X_0}}^\ell\t T_{\si_{X_1}}^\ell$ that maps $D_Y$ to $(D_{X_0},D_{X_1})$ if \eqref{bio2eq9} holds is a surjective map with contractible fibers, so by Theorem~\ref{bio3thm1}(c) the induced isomorphism $\rO_\ell[D_Y]$ of principal bundles is independent of the choice of $D_Y\in r^{-1}(D_{X_0},D_{X_1}),$ so there is a unique isomorphism $\rO_\ell[\si_Y]$ for which
\[
 \begin{tikzcd}
 	\rO_\ell(\si_{X_0})\dar{\eqref{bio3eq6}}\rar{\rO_\ell[\si_Y]} & \rO_\ell(\si_{X_1})\dar{\eqref{bio3eq6}}\\
 	\rO_\ell(D_{X_0})\rar{\rO_\ell[D_Y]} & \rO_\ell(D_{X_1})
 \end{tikzcd}
\]
commutes for all $D_{X_0}\in T_{\si_{X_0}}^\ell,$ $D_{X_1}\in T_{\si_{X_1}}^\ell,$ and $D_Y\in T_{\si_Y}^{\ell+1}$ with $r(D_Y)=(D_{X_0},D_{X_1}).$ Moreover, $\rO_\ell[\si_Y]$ depends only on the bordism class of $(Y,\si_Y).$
\end{dfn}

\subsection{Applications to gauge theory}
\label{bio32}

The following constructions are important for the applications of Theorem~\ref{bio3thm1} to orientation problems for moduli spaces in gauge theory. They also describe the relationship between the `ordinary' bordism category with tangential structures and our elliptic bordism category.

\begin{dfn}
\label{bio3def2}
Let $D_X\colon\Ga^\iy(E_0)\to\Ga^\iy(E_1)$ be a first order $\K$-linear differential operator on $X,$ where $\K=\R,$ $\C,$ or $\H.$ Let $G$ be a Lie group with a fixed representation $\rho\colon G\to\End_\K(V).$ Let $P\to X$ be a principal $G$-bundle with a connection $\na_P.$ The associated vector bundle $V_P=(P\t V)/G\to X$ has an induced connection $\na_{V_P}.$ The \emph{twisted differential operator} $D_X^{\na_{V_P}}$ is defined as
\ea*
\Ga^\iy\bigl(E_0\ot_\K V_P\bigr)&\longra\Ga^\iy\bigl(E_1\ot_\K V_P\bigr),\\
f\ot v&\longmapsto D_X(f)\ot v+\si(f)\ot\na_{V_P}(v).
\ea*

If $D_X$ is an $\ell$-adapted elliptic differential operator, then $D_X^{\na_{V_P}}$ is again an $\ell$-adapted elliptic differential operator. By letting $\na_P$ vary we obtain from $D_X$ an $\cA_P$-family $\cD_X^P$ of $\ell$-adapted elliptic differential operators parameterized by the Fr\'echet space of all connections $\cA_P$ on $P.$
\end{dfn}

\begin{rem}
In gauge theory, $\K=\R$ and $\rho=\Ad$ is the adjoint representation of $G$ on its Lie algebra. A similar construction with $\K=\H$ will be important in \cite{JUfurther} for the applications of Theorem~\ref{bio3thm1} to calibrated geometry.
\end{rem}

Let $B\O=\lim\limits_{\longrightarrow} B\O(n)$ be the classifying space for stable real vector bundles. As in Joyce--Upmeier~\cite{JUother}*{Def.~2.18}, a \emph{tangential structure} $\bs B=(B,\be)$ is a topological space $B$ with a continuous map $\be\colon B\to B\O.$ If $X$ is a smooth manifold, possibly with corners, the stable tangent bundle is classified by a map $\phi_X\colon X\to B\O.$ A \emph{$\bs B$-structure} on $X$ is represented by a continuous map $\ga_X\colon X\to B$ and a homotopy $\eta_X\colon\ga_X\circ\be\simeq \phi_X.$ In this paper, only the example $B=B\Spin=\lim\limits_{\longrightarrow} B\Spin(n)$ will be important. We now define bordism categories with tangential $\bs B$-structures.

\begin{dfn}[Joyce--Upmeier~\cite{JUother}]
\label{bio3def3}
The $\bs B$-bordism category of principal $G$-bundles $\Bord_n^{\bs B}(BG)$ has as objects all triples $(X,\ga_X,P)$ of a compact $n$-manifold $X$ with $\bs B$-structure $\ga_X$ and a principal $G$-bundle $P\to X.$ Morphisms $[Y,\ga_Y,Q]$ from $(X_0,\ga_{X_0},P_0)$ to $(X_1,\ga_{X_1},P_1)$ are equivalence classes of principal $G$-bundles $Q\to Y$ over a bordism $Y$ from $X_0$ to $X_1$ with $\bs B$-structure $\ga_Y$ and an isomorphism $\io_Y^*(Q)\cong (P_0\t[0,\varepsilon_0))\amalg (P_1\t(\varepsilon_1,1]),$ where the collar $\io_Y\colon(X_0\t[0,\varepsilon_0))\amalg(X_1\t(\varepsilon_1,1])\to Y$ preserves $\bs B$-structures. The equivalence relation is as in Definition~\ref{bio2def6}.

Disjoint union defines a monoidal structure on $\Bord_n^{\bs B}(BG)$ with the empty set as the unit object and the obvious symmetry isomorphisms, so $\Bord_n^{\bs B}(BG)$ is a symmetric monoidal category. Moreover, every object in $\Bord_n^{\bs B}(BG)$ has a \emph{dual object} $(X,-\ga_X,P).$ Every morphism is invertible due to the strong equivalence relation placed on morphisms, so $\Bord_n^{\bs B}(BG)$ is a Picard groupoid.
\end{dfn}

We could define a variant of $\Bord_n^{\bs B}(BG)$ by replacing the equivalence relation by that of diffeomorphism, leading to a category with non-invertible morphisms. This would be useful for dealing with cases where there are natural differential operators in dimensions $n$ and $n+1,$ but not in dimension $n+2.$

The following construction is similar to that of Definition~\ref{bio3def1}.

\begin{dfn}
\label{bio3def4}
Let $\bs B=\bs{Spin}.$ The real Dirac operator $D_X$ on a compact spin $n$-manifold $X$ as in Definition~\ref{bioBdef3} is an $n$-adapted elliptic differential operator on $X$ (so now $\ell=n$). Let $P\to X$ be a principal $G$-bundle. Definition~\ref{bio3def2} defines an $\cA_P$-family $\cD_X^P$ of $n$-adapted twisted real Dirac operators, so we have a $\Ga_n$-graded principal $\Ga_{n+1}$-bundle $\rO_n(\cD_X^P)\to\cA_P.$ Since $\cA_P$ is contractible, $\rO_n^{\bs B,G}(X,P)=\Ga^\iy(\rO_n(\cD_X^P))$ is a nonempty $\Ga_n$-graded $\Ga_{n+1}$-torsor, the \emph{orientation torsor} of $(X,P).$ Evaluation at $\na_P\in\cA_P$ defines an isomorphism
\e
\label{bio3eq7}
 \rO_n^{\bs B,G}(X,P)\overset{\smash\cong}{\longra}\rO_n(D_X^{\na_P}).
\e 
\end{dfn}

Also, every compact spin $(n+1)$-bordism $(Y,Q)$ defines an $\cA_Q$-family of $(n+1)$-adapted operators.

 We have the following corollary of Theorem~\ref{bio3thm1}.

\begin{cor}
\label{bio3cor1}
Let\/ $\bs B=\bs{Spin}$ and let\/ $G$ be a Lie group. There is a symmetric monoidal functor
\e
\label{bio3eq8}
 \rO_n^{\bs B,G}\colon\Bord_n^{\bs B}(BG)\longra\Ga_n/\!\!/\Ga_{n+1}
\e
that maps an object\/ $(X,P)$ to its orientation torsor as in Definition~\textup{\ref{bio3def4}}. For every morphism\/ $[Y,Q]\colon(X_0,P_0)\to(X_1,P_1)$ in\/ $\Bord_n^{\bs B}(BG)$ the diagram
\[
\begin{tikzcd}[column sep=huge]
	\rO_n^{\bs B,G}(X_0,P_0)\dar{\eqref{bio3eq7}}\rar{\rO_n^{\bs B,G}[Y,Q]} & \rO_n^{\bs B,G}(X_1,P_1)\dar{\eqref{bio3eq7}}\\
	\rO_n\bigl(D_{X_0}^{\na_{P_0}}\bigr)\rar{\rO_n\bigl[D_Y^{\na_Q}\bigr]} & \rO_n\bigl(D_{X_1}^{\na_{P_1}}\bigr)
\end{tikzcd}
\]
commutes for all\/ $\na_Q\in\cA_Q,$ $\na_{P_0}\in\cA_{P_0},$ $\na_{P_1}\in\cA_{P_1}$ with\/ $\na_Q|_{\6 Y}\cong\na_{P_0}\amalg \na_{P_1}.$
\end{cor}

Due to the pattern of groups in Table~\ref{bio2tab1}, there is only interesting new content to Corollary~\ref{bio3eq8} for $n\equiv 0,1,3,7.$ For $n\equiv 2,4,$ the theorem boils down to the bordism invariance of the $\hat A$-genus, and it is an empty statement for $n\equiv5,6.$
\medskip

The functor \eqref{bio3eq8} has an important property under change of group. Observe that a Lie group morphism $\varphi\colon G\to H$ induces a functor
\[
 F_\varphi\colon\Bord_n^{\bs B}(BG)\longra\Bord_n^{\bs B}(BH),
 \quad (X,P)\longmapsto (P\t H)/G.
\]
As in Joyce--Tanaka--Upmeier~\cite{JTU}, we call $\varphi$ of \emph{complex type} if the induced map of Lie algebras $\varphi_*\colon\g\to\h$ is injective and the quotient representation $\m=\h/\varphi_*(\g)$ has an equivariant complex structure. Observe that for all $n$ there are obvious group morphisms $\Ga_n\to\Z_2$ (non-trivial if $n\equiv0,1,2,4$) which induce functors $\Ga_n/\!\!/\Ga_{n+1}\to\Z_2/\!\!/\Z_2.$ Using the results from \cite{JTU}*{\S 2.2} it is easy to show the following.

\begin{prop}
\label{bio3prop1}
\hangindent\leftmargini
\textup{\bf(a)}\hskip\labelsep
If\/ $\varphi\colon G\to H$ is a Lie group morphism of complex type, then the following diagram commutes up to natural isomorphism.
\[
\begin{tikzcd}
 \Bord_n^{\bs B}(BG)\arrow[d,"\rO_n^{\bs B,G}"]\arrow[rr,"F_\varphi"]
 \arrow[rrd,shorten <=15pt,shorten >=15pt,Rightarrow,"\mu_{n,G}^{\bs B,H}"]
  & & \Bord_n^{\bs B}(BH)\arrow[d,"\rO_n^{\bs B,H}"]\\
 \Ga_n/\!\!/\Ga_{n+1}\arrow[r] & \Z_2/\!\!/\Z_2 & \Ga_n/\!\!/\Ga_{n+1}\arrow[l]
\end{tikzcd}
\]
\begin{enumerate}
\stepcounter{enumi}
\item
Let\/ $G_1,$ $G_2$ be Lie groups and\/ $\pi_1,$ $\pi_2$ be the projections onto the factors of\/ $G_1\t G_2.$ The following diagram commutes up to natural isomorphism.
\[
\begin{tikzcd}[column sep=huge]
 \Bord_n^{\bs B}(B(G_1\t G_2))
 \arrow[rd,Rightarrow,shorten <=15pt,shorten >=15pt,"\rho_{n,G_1,G_2}^{\bs B}"]
 \arrow[d,"{(F_{\pi_1},F_{\pi_2})}"]\arrow[r,"{\rO_n^{\bs B,G_1\t G_2}}"] & \Ga_n/\!\!/\Ga_{n+1}\\
 \Bord_n^{\bs B}(BG_1)\t\Bord_n^{\bs B}(BG_1)\arrow[r,"{\rO_n^{\bs B,G_1}\t\rO_n^{\bs B,G_2}}"] & \Ga_n/\!\!/\Ga_{n+1}\t\Ga_n/\!\!/\Ga_{n+1}\arrow[u,"\ot"]
\end{tikzcd}
\]
\end{enumerate}
\end{prop}

\section{Categorical index theory and boundary conditions}
\label{bio4}

Differential operators on manifolds with boundary are not Fredholm operators unless suitable boundary conditions are imposed. Usually there is no canonical choice of boundary condition, even up to deformation. In this section we will show that the orientation torsor $\rO_\ell(\cD_X),$ a kind of categorification of the usual index, parameterizes deformation classes of `nearly APS' boundary conditions. This will then be used to prove Theorem~\ref{bio3thm2}.

\subsection{Background on boundary conditions}
\label{bio41}

The framework of elliptic boundary conditions from Ballmann--Bär~\cite{BaBa1} proves Fredholm results for a large class of boundary conditions. After reviewing these briefly, we prove some additional elementary results, Propositions~\ref{bio4prop1} and \ref{bio4prop2}.

Let $Y$ be a Riemannian manifold with boundary $\6 Y\cong X$. Let $F_0, F_1\to Y$ be vector bundles over $\K$ with metrics on the fibers, and let $D_Y\colon\Ga^\iy(F_0)\to \Ga^\iy(F_1)$ be a $\K$-linear elliptic differential operator on $Y.$ Suppose that over a collar neighborhood $\jmath_X\colon X\t[0,T)\to Y$ of the boundary the operator $D_Y$ has the form
\e
\label{bio4eq1}
 \jmath_X^*(D_Y)=\si\bigl(\6/\6 t+A_X\bigr),
\e
where $\si\colon F_0|_X\overset{\smash\cong}{\longra} F_1|_X$ is orthogonal and $A_X\colon\Ga^\iy(F_0|_X)\to\Ga^\iy(F_0|_X)$ is a self-adjoint elliptic differential operator on $X.$ Note that $\6/\6 t$ points \emph{inward} and $\si_{\d t}(D_Y)=\si$ is the principal symbol. Since $X$ is compact without boundary, standard elliptic theory implies that $A_X$ has a discrete spectrum consisting entirely of eigenvalues Lawson--Michelsohn~\cite{LaMi}*{Ch.~III, \S 5}.

An \emph{elliptic boundary condition} for $D_Y$ is a subspace $B\subset\Ga_{L^2}(F_0|_X)$ of a certain special form as defined in \cite{BaBa1}*{Def.~7.5}.

Let $\cR\colon\Ga^\iy(F_0)\to\Ga^\iy(F_0|_X)$ be the restriction to the boundary.

\begin{thm}[{see~\cite{BaBa1}*{Cor.~8.6}}]
\label{bio4thm1}
If\/ $D_Y$ is an elliptic differential operator on a compact Riemannian manifold with boundary, of the form \eqref{bio4eq1} over a collar, and if\/ $B$ is an elliptic boundary condition for\/ $D_Y,$ then
\e
\label{bio4eq2}
 \operatorname{dom}(D_{Y,B})=\left\{\,\psi\in\Ga^\iy(F_0)\;\middle|\; \cR(\psi)\in B \,\right\}
 \xrightarrow{D_{Y,B}}
 \Ga^\iy(F_1)
\e
is an unbounded Fredholm operator. The Fredholm index of \eqref{bio4eq2} can be expressed in terms of the adjoint boundary condition\/ $B^\adj=\si(B)^\perp$ as
\e
\label{bio4eq3}
 \ind D_{Y,B}=\dim_\K\bigl(\Ker D_{Y,B}\bigr)-\dim_\K\bigl(\Ker D_{Y,B^\adj}^*\bigr).
\e
\end{thm}

For example, for every $\de\in\R$ we have the \emph{APS boundary condition}
\[
 B_\APS(\de)=\Eig_{(-\iy,\de)}(A_X).
\]

APS boundary conditions are equivalent to exponential decay conditions on the extended manifold $\widehat Y=Y\amalg_X[0,\iy)$ with non-critical decay rate as in Lockhart--McOwen \cite{LoMc}.

To turn a formally self-adjoint or skew-adjoint elliptic differential operator on a manifold with boundary into a self-adjoint or skew-adjoint Fredholm operator \eqref{bio4eq2}, it is important to work with self-adjoint boundary conditions $B^\adj=B.$

Decompose $Y$ as $Y^b\amalg Y^c$ where every connected component of $Y^b$ has a nonempty boundary and $Y^c$ is a closed manifold with $\6 Y^c=\emptyset.$ Hence $X=\6 Y=\6 Y^b.$ Decompose also $D_Y=D_{Y^b}\op D_{Y^c},$ where $D_{Y^b}\colon\Ga^\iy(F_0|_{Y^b})\to \Ga^\iy(F_1|_{Y^b})$ and $D_{Y^c}\colon\Ga^\iy(F_0|_{Y^c})\to \Ga^\iy(F_1|_{Y^c}),$ where $\Ga^\iy(F_j)$ for $j=0,1$ is a direct sum of $\Ga^\iy(F_j|_{Y^c})$ and $\Ga^\iy(F_j|_{Y^b}),$ viewed as subsets of $\Ga^\iy(F_j)$ by the extension of sections by zero. By the Unique Continuation Theorem of Aronszajn~\cite{Aro}, the restriction of $\cR$ to $\Ker D_{Y^b}\subset\Ga^\iy(F_0|_{Y^b})$ is injective.

\begin{prop}
\label{bio4prop1}
$\Ker D_{Y,B_\APS(-\de)}=\Ker D_{Y^c}$ for all sufficiently large\/ $\de.$
\end{prop}

\begin{proof}
Firstly, if $\varepsilon>\de,$ then $\Ker D_{Y,B_\APS(-\varepsilon)}\subset\Ker D_{Y,B_\APS(-\de)}.$ Secondly, we have $\bigcap_{\de\in\R}\Ker D_{Y,B_\APS(-\de)}=\Ker D_{Y^c}$ because if $\psi\in\Ker D_Y$ satisfies $\cR(\psi)\in B_\APS(-\de)\allowbreak=\Eig_{(-\iy,\de)}(A_X)$ for all $\de\in\R,$ then $\cR(\psi)=0$ and thus $\psi|_{Y^b}=0$ by the Unique Continuation Theorem. Of course, $\Ker D_{Y^c}\subset\Ker D_{Y,B_\APS(-\de)}$ for all $\de$ and the previous two facts imply that the dimensions of these subspaces are the same for $\de$ sufficiently large, so these subspaces are in fact equal.
\end{proof}

\begin{prop}
\label{bio4prop2}
Let\/ $B_1,$ $B_2$ be elliptic boundary conditions for\/ $D_Y$ with\/ $B_1\subset B_2$ and\/ $\dim_\K B_1^\perp\cap B_2<\iy.$ Let\/ $\pi\colon\Ga_{L^2}(F_0|_X)\to B_1^\perp\cap B_2$ be the orthogonal projection. The maps\/ $\pi\cR|_{\Ker D_{Y,B_2}},$ $\pi\si^*\cR|_{\Ker D_{Y,B_1^\adj}^*}$ factor through embeddings
\ea
 \Ker D_{Y,B_2}/\Ker D_{Y,B_1}&\xrightarrow{\pi\cR}B_1^\perp\cap B_2,\label{bio4eq4}\\
 \Ker D_{Y,B_1^\adj}^*/\Ker D_{Y,B_2^\adj}^*&\xrightarrow{\pi\si^*\cR} B_1^\perp\cap B_2,\label{bio4eq5}
\ea
whose images are orthogonal complements of each other. In particular,
\e
\begin{aligned}
 B_1^\perp\cap B_2={}&\pi\cR\Bigl(\Ker D_{Y,B_2}/\Ker D_{Y,B_1}\Bigr)\\
 &\op\pi\si^*\cR\Bigl(\Ker D_{Y,B_1^\adj}^*/\Ker D_{Y,B_2^\adj}^*\Bigr).
\end{aligned}
\label{bio4eq6}
\e
\end{prop}

\begin{proof}
From $B_1\subset B_2$ we have $B_2=B_1\op(B_1^\perp\cap B_2).$ Suppose $\pi\cR(\psi)=0$ with $\psi\in\Ker D_{Y,B_2}.$ Then $D_Y(\psi)=0,$ $\cR(\psi)\in B_2,$ and $\cR(\psi)\perp(B_1^\perp\cap B_2).$ Hence $\cR(\psi)\in B_1$ and so $\psi\in\Ker D_{Y,B_1}.$ This shows that $\pi\cR|_{\Ker D_{Y,B_2}}$ factors through an embedding \eqref{bio4eq4}; the proof for \eqref{bio4eq5} is similar.

Let $\psi_1\in\Ker D_{Y,B_1^\adj}^*,$ $\psi_2\in\Ker D_{Y,B_2}.$ Then $D_Y^*(\psi_1)=0,$ $\si^*\cR(\psi_1)\in B_1^\perp,$ $D_Y(\psi_2)=0,$ and $\cR(\psi_2)\in B_2.$ By Green's formula \cite{BaBa1}*{Lem.~2.6} we have
\[
 0=\int_Y\ban{D_Y^*(\psi_1),\psi_2}+\int_Y\ban{\psi_1,D_Y(\psi_2)}=\int_X\ban{\cR(\psi_1),\si(\cR(\psi_2))}.
\]
Thus $\an{\si^*\cR(\psi_1),\cR(\psi_2)}_{L^2}=0$ and since $\si^*\cR(\psi_1)\in B_1,$ $\cR(\psi_2)\in B_2^\perp$ the inner product can also be evaluated after taking the orthogonal projection,
\[
 \ban{\pi\si^*\cR(\psi_1),\pi\cR(\psi_2)}_{B_1^\perp\cap B_2}=\ban{\si^*\cR(\psi_1),\cR(\psi_2)}_{L^2}=0.
\]
It follows that the images of the embeddings \eqref{bio4eq4}, \eqref{bio4eq5} are perpendicular.

For elliptic boundary conditions $B_1\subset B_2,$ \cite{BaBa1}*{Cor.~8.8} shows that $B_2/B_1$ is finite-dimensional and
\e
\label{bio4eq7}
 \dim_\K B_2/B_1=\ind D_{Y,B_2}-\ind D_{Y,B_1}.
\e
Now \eqref{bio4eq3} and \eqref{bio4eq7} show that the dimension of the sum of the images of the embeddings \eqref{bio4eq4}, \eqref{bio4eq5} agrees with the dimension of the ambient space $B_1^\perp\cap B_2,$ so the images are indeed complementary subspaces.
\end{proof}

\subsubsection{Transmission conditions}
\label{bio411}

In preparation of the proof of Theorem~\ref{bio3thm2}(a), suppose that $Y$ contains a cylinder $\jmath_{X_0}\colon X_0\t(-\varepsilon,\varepsilon)\to Y$ over which $D_Y$ takes the form \eqref{bio4eq1}. Let $Y'$ be the manifold with boundary obtained by cutting $Y$ open along $X_0$ so that $\6 Y'$ consists of $\6 Y$ and two copies $X_\pm$ of $X_0.$ Let $F_0', F_1'\to Y'$ be the pullback vector bundles and let $D_{Y'}$ be the pullback differential operator on $Y'.$ For the collars $\jmath_{X_\pm}\colon X_\pm\t[0,\varepsilon)\to Y',$ $(x,t)\mapsto \jmath_{X_0}(x,\pm t)$ of $X_\pm$ we have
\e
\label{bio4eq8}
 \jmath_{X_+}^*(D_{Y'})=\si\left(\frac{\6}{\6 t}+A_{X_0}\right),\quad
 \jmath_{X_-}^*(D_{Y'})=(-\si)\left(\frac{\6}{\6 t}-A_{X_0}\right),
\e
and hence $A_{X_\pm}=\pm A_{X_0}$ where we identify $F_0'|_{X_\pm}\cong F_0|_{X_0}.$ Observe that $\si$ changes sign over $X_-$ by our convention that $\6/\6 t$ points inwards.

The \emph{transmission} condition as in \cite{BaBa1}*{Ex.~7.28} is
$
 B_\trans(X_0)=\{
 (\psi_+,\psi_-)\in\Ga^\iy(F_0|_{X_+})\op\Ga^\iy(F_0|_{X_-})
 \mid
 \psi_-=\psi_+
 \}.
$
If $\6 Y=\es,$ then $B_\trans(X_0)$ is an elliptic boundary condition over $\6 Y'=X_+\amalg X_-.$ (In general, we should combine $B_\trans(X_0)$ with an arbitrary elliptic boundary condition $B$ over $\6 Y$ to get an elliptic boundary condition $B\op B_\trans(X_0)$ over $\6 Y'=\6 Y\amalg X_+\amalg X_-.$)
%
For all $\de\ge 0$ there is a continuous deformation of elliptic boundary conditions
\e
\label{bio4eq9}
 B_\APS(-\de)\op\Ga\bigl(\id_{\Eig_{[-\de,\de]}(A_{X_0})}\bigr)\simeq B_\trans(X_0).
\e

\subsection{Cauchy data and nearly APS boundary conditions}
\label{bio42}

This section introduces the concept of a nearly APS boundary condition.

\begin{dfn}
\label{bio4def1}
The subspace of \emph{Cauchy data} for the operator $D_Y$ is
\[
 C(D_Y)
 =\cR(\Ker D_Y)
 =
 \left\{
 \phi\in\Ga^\iy(F_0|_X)
 \;\middle|\;
 \exists \psi\in\Ker D_Y: \cR(\psi)=\phi
 \right\}.
\]
\end{dfn}

If $B$ is an elliptic boundary condition, then $\cR$ restricts to an isomorphism
\e
\label{bio4eq10}
 \cR|_{\Ker D_{Y^b,B}}\colon\Ker D_{Y^b,B}\overset{\cong}{\longra}C(D_Y)\cap B.
\e

In particular, Theorem~\ref{bio4thm1} implies that $C(D_Y)\cap B$ is finite-dimensional. It is sometimes possible to replace this intersection of infinite-dimensional subspaces by an intersection inside a finite-dimensional space.

\begin{dfn}
\label{bio4def2}
A boundary condition is \emph{nearly APS} if it can be written as $B=B_\APS(-\de)\op L_\de$ for some $\de\ge0$ and a finite-dimensional subspace $L_\de\subset\Eig_{[-\de,\de]}(A_X).$ For all $0\le\de<\varepsilon$ it is then again possible to write $B=B_\APS(-\varepsilon)\op L_\varepsilon$ with $L_\varepsilon=\Eig_{[-\varepsilon,-\de)}(A_X)\op L_\de.$
\end{dfn}

\begin{dfn}
\label{bio4def3}
Let $\pi_\de$ be the orthogonal projection of $\Ga_{L^2}(F_0|_X)$ onto the closed subspace $\Eig_{[-\de,\de]}(A_X).$ The space of \emph{$\de$-projected Cauchy data} for $D_Y$ is
\[
 C_\de(D_Y)=\left\{\varphi\in\Eig_{[-\de,\de]}(A_X) \;\middle|\; \exists\psi\in\Ker D_Y: \cR(\psi)-\varphi\in B_\APS(-\de)\right\}.
\]
\end{dfn}

\begin{prop}
\label{bio4prop3}
Let\/ $B$ be a nearly APS boundary condition, written as\/ $B=B_\APS(-\de)\op L_\de$ for large\/ $\de\ge0.$ Then for all sufficiently large\/ $\de\ge 0$ the map\/ $\pi_\de\circ\cR\colon\Ga^\iy(F_0)\to\Eig_{[-\de,\de]}(A_X)$ restricts to an embedding of\/ $\Ker D_{Y^b,B}$ into\/ $\Eig_{[-\de,\de]}(A_X).$ For all\/ $\de\ge 0$ the image is\/ $\pi_\de(C(D_Y)\cap B)=C_\de(D_Y)\cap L_\de.$
\end{prop}

\begin{proof}
The kernel of $\pi_\de\cR|_{\Ker D_{Y^b,B}}$ is the set of all $\psi\in\Ga^\iy(F_0|_{Y^b})$ satisfying $D_Y(\psi)=0,$ $\cR(\psi)\in B,$ and $\cR(\psi)\in\Eig_{(-\iy,-\de)\cup(\de,+\iy)}(A_X).$ For large $\de$ we have $B=B_\APS(-\de)\op L_\de$ and thus $B\cap \Eig_{(-\iy,-\de)\cup(\de,+\iy)}(A_X)=B_\APS(-\de).$ By Proposition~\ref{bio4prop1}, $B_\APS(-\de)\cap\Ker D_{Y^b,B}=\{0\}$ for all sufficiently large $\de,$ hence $\pi_\de\cR|_{\Ker D_{Y^b,B}}$ is injective for large $\de,$ with image $\pi_\de(C(D_Y)\cap B),$ see \eqref{bio4eq10}.

Next, we identify the image. Clearly $\pi_\de(C(D_Y)\cap B)\subset \pi_\de(C(D_Y))\cap \pi_\de(B)\subset C_\de(D_Y)\cap L_\de.$ Conversely, if $\varphi\in C_\de(D_Y)\cap L_\de$ then there is $\psi\in\Ker D_Y$ such that $\cR(\psi)-\varphi\in B_\APS(-\de).$ Therefore $\cR(\psi)\in\varphi+B_\APS(-\de)\subset L_\de\op B_\APS(-\de)=B$ and hence $\cR(\psi)\in C(D_Y)\cap B.$ Moreover, $\pi_\de\cR(\psi)=\pi_\de(\varphi+B_\APS(-\de))=\varphi,$ thus showing $\varphi\in \pi_\de(C(D_Y)\cap B).$
\end{proof}

If follows that for a nearly APS boundary condition and sufficiently large $\de$ we may equivalently take the intersection $C(D_Y)\cap B$ in the \emph{finite-dimensional} space $\Eig_{[-\de,\de]}(A_X)$ in the sense that of an isomorphism $\pi_\de\colon C(D_Y)\cap B\to C_\de(D_Y)\cap L_\de.$

\begin{prop}
\label{bio4prop4}
If\/ $D_Y$ is formally self-adjoint or formally skew-adjoint, then there is an orthogonal decomposition\/ $\Eig_{[-\de,\de]}(A_X)=C_\de(D_Y)\op\si(C_\de(D_Y)).$
\end{prop}

\begin{proof}
Putting the condition $D_Y^*=\pm D_Y$ into \eqref{bio4eq1} shows $\si^*=\si^{-1}=\pm\si$ and $A_X\si^*=\pm\si A_X,$ hence $A_X$ and $\si$ anti-commute in both cases. In particular, $\si$ preserves $\Eig_{[-\de,\de]}(A_X)$ and commutes with the projection $\pi_\de$ onto this subspace. Moreover, for $B_1=\Eig_{(-\iy,-\de)}(A_X),$ $B_2=\Eig_{(-\iy,\de]}(A_X)$ we have $B_1^\adj=B_2,$ $B_2^\adj=B_1.$ Now put this into Proposition~\ref{bio4prop2}, observing that $B_1^\perp\cap B_2=\Eig_{[-\de,\de]}(A_X),$ $\pi=\pi_\de,$ and that by the description of the image in Proposition~\ref{bio4prop3} applied to $D_{Y^b}$ and $B=B_2$ the subspace that appears in \eqref{bio4eq6} is $
 \pi\cR\bigl(\Ker D_{Y^b,B_2}/\Ker D_{Y^b,B_1}\bigr)=C_\de(D_Y).$
\end{proof}

\subsection{\texorpdfstring{Proof of Theorem~\ref{bio3thm2}}{Proof of Theorem 3.2}}
\label{bio43}

We now begin the proof of the main result. Note that formula \eqref{bio3eq3} can be verified at each $t\in T.$ Moreover, we can construct the isomorphism \eqref{bio3eq4} over each $t\in T$ individually and then prove continuity. For this it is important to allow non-zero cut-offs $\de\ge 0$ in the construction because the zero eigenspace by itself would have discontinuous jumps.

\begin{dfn}
\label{bio4def4}
Let $D_X$ and $D_Y$ be $\ell$-adapted and $(\ell+1)$-adapted elliptic differential operator on $X$ and $Y,$ respectively, and suppose that $\jmath_X^*(D_Y)\cong\Cyl(D_X)$ over a collar neighborhood. Observe that for every $\ell$ we can write $\jmath_X^*(D_Y)$ in the form \eqref{bio4eq1} with $A_X$ self-adjoint. For all $\de\ge 0$ define
\[
 W_\de(D_X)=\Eig_{[-\de,\de]}(A_X)
\]
which we equip with the non-degenerate bilinear form
\e
\be_\de(\psi_1,\psi_2)=\int_X\ban{\si_{\d t}(D_Y)(\psi_1),\psi_2}.
\label{bio4eq11}
\e
Varying the cut-off, we find that for all $0\le\de<\varepsilon$ there are isomorphisms
\e
\label{bio4eq12}
 \Lag W_\de(D_X)\longra\Lag W_\varepsilon(D_X),
\quad
L_\de\longmapsto L_\de\op\Eig_{[-\varepsilon,-\de)}(A_X).
\e
\end{dfn}

The `quadratic space' $W_\de(D_X)$ plays a central role in the proof of Theorem~\ref{bio3thm2}. For example, in the context of the transmission condition the diagonal subspace $B_\trans(X_0)\cap W_\de(D_X)=\Ga(\id_{\Eig_{[-\de,\de]}(A_{X_0})})$ is `isotropic' (note that $\si$ changes sign over $X_-$ in \eqref{bio4eq8}). Since the exact type of $\be_\de$ depends on $\ell$ mod $8,$ we proceed case-by-case.

\subsubsection{\texorpdfstring{Proof of Theorem~\textup{\ref{bio3thm2}} in the case $\ell\equiv 0$}{Proof of Theorem 3.2 in the case ℓ≡0}}

\label{bio431}

According to the definition of $\ell$-adapted and \eqref{bio2eq2}, $D_X\colon\Ga^\iy(E_0)\to\Ga^\iy(E_1)$ is an $\R$-linear differential operator, $F_0|_{\6 Y}=F_1|_{\6 Y}=E_0\op E_1,$ and
\e
\label{bio4eq13}
 \Cyl(D_X)=\si
 \left(\frac{\6}{\6 t}+A_X\right),
 \quad
 \si
 =
 \begin{pmatrix}
 1 & 0\\
 0 & -1	
 \end{pmatrix},
 \quad
 A_X=
 \begin{pmatrix}
 0   & D_X^*\\
 D_X & 0
 \end{pmatrix}.
\e

Hence $W_\de(D_X)$ is a real vector space and $\be_\de$ is a non-degenerate symmetric $\R$-bilinear form if $\ell\equiv0.$

\subsubsection*{Real symmetric bilinear forms}

A non-degenerate symmetric $\R$-bilinear form $\be$ on a finite-dimensional real vector space $W$ is uniquely determined up to isomorphism by its dimension and its inertia index. A subspace $L\subset W$ is \emph{Lagrangian} if $L^{\perp_\be}=L.$ By Meinrenken~\cite{Mein13}*{Thm.~1.3}, the set $\Lag W$ of Lagrangian subspaces is either empty or a homogeneous space, isomorphic to an orthogonal group $\O(L).$ By convention, $\Lag W=\{\pm1\}$ if $W=\{0\}.$ Lagrangian subspaces exist if and only if $\dim_\R W$ is even and the inertia index vanishes.

Let $V^\pm$ be Euclidean vector spaces and define $W=V^+\op V^-$ and $\be=\an{,}_{V^+}-\an{,}_{V^-}.$ The inertia index of $W$ is $\dim_\R V^+-\dim_\R V^-,$ so if $\dim V^+=\dim V^-,$ then $\Lag W\neq\emptyset.$ Observe that every Lagrangian subspace of $W$ is the graph $L=\Ga(f)$ of a unique orthogonal isomorphism $f\colon V^+\to V^-$; for the intersection of graphs we have $\Ga(f_1)\cap\Ga(f_2)=\Eig_{+1}(f_1f_2^{-1}).$ From this we find that if $L\in\Lag W$ and $f\colon V^+\to V^-,$ $g\colon V^+\to V^+$ are orthogonal isomorphisms, then
\e
\label{bio4eq14}
 (-1)^{\dim_\R\Ga(f)\cap L}=\det(g)(-1)^{\dim_\R\Ga(fg)\cap L}.
\e
As $\det(g)\in\{\pm1\}$ we see in particular that the left-hand side of \eqref{bio4eq14} only depends on the path-component of $f.$

%
%
%

\subsubsection*{Proof of bordism invariance \textup{\eqref{bio3eq3}}}

The spectral properties of $D_X$ and $A_X$ are closely linked: all non-zero eigenvalues of the self-adjoint Fredholm operators $D_X^*D_X$ and $D_XD_X^*$ agree with multiplicity as there are orthogonal isomorphisms
\e
\label{bio4eq15}
 \pm\la^{-1}D_X\colon\Eig_{\la^2}(D_X^*D_X)\longra\Eig_{\la^2}(D_XD_X^*)
\e
for all $\la>0$ whose graphs are the eigenspaces
\e
\label{bio4eq16}
 \Ga(\pm\la^{-1}D_X)=\Eig_{\pm\la}(A_X).
\e
In particular, \eqref{bio4eq11} decomposes as $W_\de(D_X)=V_{[0,\de]}^+(D_X)\op V_{[0,\de]}^-(D_X),$ where
\ea*
 V_{[0,\de]}^+(D_X)&=\Eig_{[0,\de^2]}(D_X^*D_X),
&V_{[0,\de]}^-(D_X)&=\Eig_{[0,\de^2]}(D_XD_X^*),
\ea*
and the inertia index of $W_\de(D_X)$ is
\ea*
\dim_\R\Eig_{[0,\de^2]}(D_X^*D_X)-\dim_\R\Eig_{[0,\de^2]}(D_XD_X^*)\overset{\smash{\eqref{bio4eq15}}}{=}\ind_\ell D_X.
\ea*
On the other hand, $C_\de(D_Y)\subset W_\de(D_X)$ is a Lagrangian subspace by Proposition~\ref{bio4prop4}, so the inertia index vanishes. This proves \eqref{bio3eq3} for $\ell\equiv 0.$\hfill\qedsymbol

\subsubsection*{Lagrangians and orientations}

We already know $\Lag W_\de(D_X)\neq\es,$ which thus has two path components. We now identify the $\Z_2$-torsor $\pi_0\Lag W_\de(D_X)$ with the $\Z_2$-torsor of orientations $O(\DET_\R D_X)$ of the determinant line of $D_X.$

Every Lagrangian subspace $L_\de\subset W_\de(D_X)$ is the graph of an orthogonal isomorphism $f_\de\colon V_{[0,\de]}^+(D_X)\to V_{[0,\de]}^-(D_X)$ (note that the `preferred' isomorphism \eqref{bio4eq15} only works for $\la>0,$ so this is a genuine choice).
Up to deformation, there are two possibilities for $f_\de,$ distinguished by the orientation $O(\det f_\de)=\R_{>0}\det f_\de$ of the determinant line $\det V_{[0,\de]}^+(D_X)\ot\det V_{[0,\de]}^-(D_X)^*$ as in \eqref{bioAeq6}. Inserting a sign, we thus obtain for all $\de\ge 0$ an isomorphisms of $\Z_2$-torsors
\ea
	\pi_0\Lag W_\de(D_X)
	&\overset{\cong}{\longra}
	O\bigl(\det V_{[0,\de]}^+(D_X)\ot\det V_{[0,\de]}^-(D_X)^*\bigr),
	\notag\\
	[L_\de]=[\Ga(f_\de)]&\longmapsto(-1)^{\dim_\R V_{[0,\de]}^+(D_X)}O(\det f_\de).
\label{bio4eq17}
\ea
These fit for all $0\le\de<\varepsilon$ into a commutative diagram (see \S\ref{bioA2} for notation)
\[
\begin{tikzcd}[row sep=tiny]
	\pi_0\Lag W_\de(D_X)\arrow[dd,"\eqref{bio4eq12}"]\rar{\eqref{bio4eq17}}&
	O\bigl(\det V_{[0,\de]}^+(D_X)\ot\det V_{[0,\de]}^-(D_X)^*\bigr)
	\arrow[dd,"{\stab_{\de,\varepsilon}(D_X)}"]
	\arrow[rd,"{\stab_{\de}(D_X)}",pos=0.3]& \\
	&& O(\DET_\R D_X),\\
	\pi_0\Lag W_\varepsilon(D_X)\rar{\eqref{bio4eq17}}&
	O\bigl(\det V_{[0,\varepsilon]}^+(D_X)\ot\det V_{[0,\varepsilon]}^-(D_X)^*\bigr)\arrow[ru,"{\stab_{\varepsilon}(D_X)}"',pos=0.3]& \\
\end{tikzcd}
\]
which justifies the sign in \eqref{bio4eq17}.
It follows that there is an isomorphism
\e
\label{bio4eq18}
 \lim\limits_{\longra}\pi_0\Lag W_\de(D_X)\overset{\smash\cong}{\longra} O(\DET_\R D_X)
\e
of $\Z_2$-torsors. For a family of $\ell$-adapted differential operators $\cD_X=\{D_X(t)\},$ \eqref{bio4eq18} becomes a continuous isomorphism of principal $\Z_2$-bundles since \eqref{bio4eq17} is continuous over $T_\de\subset T,$ using the notation of Definition~\ref{bioAdef2}.

\begin{rem}
\label{bio4rem1}
Since $\Lag W_\de(D_X)$ is isomorphic to an orthogonal group, it carries also an interesting principal $\Z_2$-bundle $P\to\Lag W_\de(D_X).$ This is one categorical level higher than the set of path components considered here. This principal $\Z_2$-bundle relates to spin structures on $D_X.$ The author hopes to give applications of this observation to $K$-theoretic enumerative invariants and to spin structures on moduli spaces in future work.
\end{rem}

\subsubsection*{Construction of \textup{\eqref{bio3eq4}}}

The isomorphism \eqref{bio4eq18} depends on the principal symbol $\si_{\d t}(D_Y)$ of $D_Y$ over the boundary, which is all we have used so far. We next use the differential operator $D_Y$ to construct a trivialization of \eqref{bio4eq18}.

The key observation is that every Lagrangian subspace $L_\de\subset W_\de(D_X)$ determines a self-adjoint nearly APS boundary condition $B=B_\APS(-\de)\op L_\de$ for the differential operator $D_Y.$ As $B$ is self-adjoint, \eqref{bio4eq2} is a skew-adjoint Fredholm operator and hence the mod-two index $(-1)^{\dim_\R\Ker D_{Y,B}}$ depends continuously on $D_{Y,B}.$ In particular, there is a well-defined isomorphism
\[
 \pi_0\Lag W_\de(D_X)\xrightarrow{\smash{i_{Y,\de}}}\{\pm1\},
 \quad [L_\de]\longmapsto(-1)^{\dim_\R\Ker D_{Y,B_\APS(-\de)\op L_\de}}.
\]
(If $X=\emptyset,$ then $\pi_0\Lag W_\de(D_X)=\{\pm1\}$ and $i_{Y,\de}$ becomes multiplication by the mod-two index $(-1)^{\dim_\R\Ker D_Y}$ on the closed manifold $Y$ for all $\de\ge0.$)

Varying the cut-off, let $0\le\de<\varepsilon$ and $L_\varepsilon$ be the image of $L_\de$ under \eqref{bio4eq12}. By \eqref{bio4eq16} we have $B_\APS(-\de)\op L_\de=B_\APS(-\varepsilon)\op L_\varepsilon$ and so $i_{Y,\de}[L_\de]=i_{Y,\varepsilon}[L_\varepsilon].$ Hence $i_{Y,\de}$ factors through the canonical maps to an isomorphism
\e
\label{bio4eq19}
 \lim\limits_{\longra}\pi_0\Lag W_\de(D_X)\overset{\smash\cong}{\longra}\{\pm1\}.
\e

\begin{dfn}
\label{bio4def5}
Let $\tau_\ell[D_Y]\colon O(\DET_\R D_X)\to\{\pm1\}$ be the composition of the inverse of \eqref{bio4eq18} with \eqref{bio4eq19}. In other words, for all $\de\ge 0$ we have
\[
 \tau_\ell[D_Y]\bigl(\R_{>0}\stab_{\de}(D_X)\det f_\de\bigr)=(-1)^{\dim_\R\Ker D_{Y,B}+\dim_\R V_{[0,\de]}^+(D_X)},
\]
where $B=B_\APS(-\de)\op\Ga(f_\de)$ and $f_\de\colon V_{[0,\de]}^+(D_X)\to V_{[0,\de]}^-(D_X)$ is orthogonal. (If $X=\emptyset,$ then $O(\DET_\R D_X)=\{\pm1\}$ and $\tau_\ell[D_Y](\pm1)=\pm(-1)^{\dim_\R\Ker D_Y}.$)
\end{dfn}

Now let $\cD_X,$ $\cD_Y$ be families of adapted elliptic differential operators as in Theorem~\ref{bio3thm2}. Then \eqref{bio4eq18} and \eqref{bio4eq19} are continuous and hence $\tau_\ell[\cD_Y]=\coprod_{t\in T}\tau_\ell[D_Y(t)]$ is a continuous isomorphism of principal $\Z_2$-bundles over $T.$

Each of the statements of Theorem~\ref{bio3thm2}(a)--(e) can be verified point-wise, so we can fix $t\in T$ and drop it from the notation; (b) has already been proven.

\subsubsection*{Proof of \textup{\ref{bio3thm2a}}}

For simplicity, assume $\6 Y=\emptyset,$ so $\6 Y'=X_+\amalg X_-.$ The general case requires straight-forward bookkeeping for the other boundary components and will be left to the reader. Continue in the notation of \S\ref{bio411}. We have
\ea*
 \jmath_{X_-}^*(D_{Y'})
 &=(-\si)\left(\frac{\6}{\6 t}-A_{X_0}\right) &&\text{by \eqref{bio4eq8}}\\
 &=
 \begin{pmatrix}
 	-\frac{\6}{\6 t} & D_{X_0}^*\\
 	-D_{X_0} & \frac{\6}{\6 t}
 \end{pmatrix}
 && \text{by \eqref{bio4eq13}}\\
 &\cong
 \begin{pmatrix}
 	\frac{\6}{\6 t} & -D_{X_0}^*\\
 	D_{X_0} & -\frac{\6}{\6 t}
 \end{pmatrix}
 && \text{conjugate by $\begin{pmatrix}0 & 1\\ 1& 0\end{pmatrix}$}\\
 &=\Cyl(-D_{X_0}^*) && \text{by \eqref{bio2eq2}}
\ea*
and $\jmath_{X_+}^*(D_{Y'})=\Cyl(D_{X_0}).$ Hence $D_{X_+}= D_{X_0},$ $D_{X_-}\cong -D_{X_0}^*,$ and $D_{X_+\amalg X_-}\cong D_{X_+}\op D_{X_-}=D_{X_0}\op -D_{X_0}^*$ are isomorphic $\ell$-adapted operators. Let $\de\ge 0.$ We can identify $V^\pm_{[0,\de]}(D_{X_0})\cong V^\pm_{[0,\de]}(D_{X_+})\cong V^\mp_{[0,\de]}(D_{X_-})$ and $V^\pm_{[0,\de]}(D_{X_+}\op D_{X_-})\cong V_{[0,\de]}^\pm(D_{X_0})\op V_{[0,\de]}^\mp(D_{X_0}).$ Let $n_\pm=\dim_\R V^\pm_{[0,\de]}(D_{X_0}).$

Let $\om_\de^\pm\in\det V_{[0,\de]}^\pm(D_{X_0})$ be volume forms and $(\om_\de^\pm)^*\in\det V_{[0,\de]}^\pm(D_{X_0})^*$ be the dual volume forms as in \eqref{bioAeq5}. The diagram \eqref{bio3eq5} involves the isomorphisms
\[
\begin{tikzcd}[column sep=-3mm]
\begin{array}{l}
\DET_\R D_{X_+}\ot \DET_\R D_{X_-}\\
\cong
\det V^+_{[0,\de]}(D_{X_+})
\ot\det V^-_{[0,\de]}(D_{X_+})^*\\
\ot\det V^+_{[0,\de]}(D_{X_-})
\ot\det V^-_{[0,\de]}(D_{X_-})^*\\
\ni\om_\de^+\ot(\om_\de^-)^*\ot\om_\de^-\ot(\om_\de^+)^*
\end{array}
\arrow[rd,"{\eqref{bioAeq10},\eqref{bioAeq11}}"]
\arrow[dd,"{\eqref{bioAeq12}}"]
&
\\
& \begin{array}{l}
\DET_\R D_{X_0}\ot(\DET_\R D_{X_0})^*
\arrow[d,"\an{\,,\,}"]\\
\cong
\det V^+_{[0,\de]}(D_{X_0})
\ot\det V^-_{[0,\de]}(D_{X_0})^*\\
\ot\det V^-_{[0,\de]}(D_{X_0})^{**}
\ot\det V^+_{[0,\de]}(D_{X_0})^*\\
\ni(-1)^{n_+}\om_\de^+\ot(\om_\de^-)^*\\\quad\ot\ev(\om_\de^-)\ot(\om_\de^+)^*
\end{array}
\\
\begin{array}{l}
\DET_\R (D_{X_+}\op D_{X_-})\\
\cong
\det V^+_{[0,\de]}(D_{X_+}\op D_{X_-})\\
\ot\det V^-_{[0,\de]}(D_{X_+}\op D_{X_-})^*\\
\ni(-1)^{n_-(n_-+n_+)}\om_\de^+\wedge\om_\de^-\\
\quad\ot(\om_\de^+)^*\wedge (\om_\de^-)^*
\end{array}
 & \R\ni(-1)^{n_+},
\end{tikzcd}
\]
which map the elements in the indicated way, where the sign in the horizontal map is due to \eqref{bioAeq11} applied to $D_-=-D_0^*.$ 

To evaluate the bottom horizontal map in \eqref{bio3eq5}, according to Definition~\ref{bio4def5}, we should write $(-1)^{n_-(n_-+n_+)}\om_\de^+\wedge\om_\de^-\ot(\om_\de^+)^*\wedge (\om_\de^-)^*$ in the form $\det f_\de$ for an orthogonal isomorphism $f_\de\colon V^+_{[0,\de]}(D_{X_+}\op D_{X_-})\to V^-_{[0,\de]}(D_{X_+}\op D_{X_-}).$ Using the identification $V^\pm_{[0,\de]}(D_{X_+}\op D_{X_-})\cong V_{[0,\de]}^\pm(D_{X_0})\op V_{[0,\de]}^\mp(D_{X_0})$ define $f_\de(v^+\op v^-)=v^-\op v^+.$ By \eqref{bioAeq6} we then have $\det(f_\de)=\om_\de^+\wedge\om_\de^-\ot(\om_\de^-)^*\wedge(\om_\de^+)^*=(-1)^{n_-n_+}\om_\de^+\wedge\om_\de^-\ot(\om_\de^+)^*\wedge (\om_\de^-)^*.$
Hence $\tau_\ell[D_{Y'}]$ maps $(-1)^{n_-(n_-+n_+)}\om_\de^+\wedge\om_\de^-\ot(\om_\de^+)^*\wedge (\om_\de^-)^*$ to $(-1)^{n_++n_-+n_-^2+\dim_\R\Ker D_{Y',B'}}=(-1)^{n_++\dim_\R\Ker D_{Y',B'}},$ where the boundary condition $B'=B_\APS(-\de)\op\Ga(f_\de)$ can by \eqref{bio4eq9} be deformed into the transmission condition, so $\dim_\R\Ker D_{Y',B'}=\dim_\R\Ker D_Y.$ Putting this calculation into the above diagram completes the proof of (a).
\hfill\qedsymbol

\subsubsection*{Proofs of \textup{\ref{bio3thm2c}}, \textup{\ref{bio3thm2d}}, \textup{\ref{bio3thm2e}}}

The proofs of parts (c) and (d) are straight-forward consequences of the fact that by \eqref{bioAeq12} determinant line bundles are compatible with direct sums and that the mod-two index is additive under direct sums.

Part (e) is obvious from Definition~\ref{bio4def5}.

\subsubsection{\texorpdfstring{Proof of Theorem~\textup{\ref{bio3thm2}} in the case $\ell\equiv 1$}{Proof of Theorem 3.2 in the case ℓ≡1}}
\label{bio432}

According to \eqref{bio2eq3}, $E_0=E_1,$ $D_X\colon\Ga^\iy(E_0)\to\Ga^\iy(E_1)$ is a skew-adjoint $\R$-linear differential operator, $F_0|_{\6 Y}=E_0\ot\C,$ $F_1|_{\6 Y}=\ol{F_0|_{\6 Y}},$ $\si$ is the complex conjugation, and (here, $(D_X)_\C$ denotes the complexification of $D_X$)
\e
\label{bio4eq20}
 D_Y
 =
 \si\left(
 \frac{\6}{\6 t}+A_X
 \right),\qquad
 A_X=i(D_X)_\C\overset{\text{over $\R$}}{=}\begin{pmatrix}
	0&D_X^*\\
	D_X & 0
\end{pmatrix}.
\e

Hence $W_\de(D_X)$ is a complex vector space, $\be_\de$ is a non-degenerate symmetric $\C$-bilinear form, and $\si$ is a `positive' real structure on $W_\de(D_X)$ if $\ell\equiv 1.$

\subsubsection*{Complex symmetric bilinear forms}

Let $\be$ be a non-degenerate symmetric $\C$-bilinear form on a finite-dimensional complex vector space $W.$ A complex subspace $L\subset W$ is \emph{Lagrangian} if $L^{\perp_\be}=L.$ Let $\Lag_\C W$ be the set of complex Lagrangian subspaces, where $\Lag_\C W=\{\pm1\}$ if $W=\{0\}.$ Hence $\Lag_\C W$ is nonempty if and only if $\dim_\C W$ is even.

A \emph{real structure} is a $\C$-linear map $\si\colon W\to\ol W$ such that $\ol\si\circ\si=\id_W.$ We call $\si$ \emph{positive} if $h(w_1,w_2)=\be(w_1,\si(w_2))$ is a Hermitian metric on $W.$ In this case $V=\Eig_{+1}(\si)$ is a real vector subspace of $W$ with Euclidean metric $g=h|_V$;
we can identify $W\cong V\ot_\R\C,$ $\dim_\C W=\dim_\R V,$ and $\be$ with the complexification of $g.$ Moreover, there is a bijection between $\Lag_\C W$ and the set $\mathcal{J}(V)=\{J\in\Hom_\R(V,V)\mid J^*=J^{-1}=J\}$ of orthogonal complex structures on $V,$ where the Lagrangian subspace corresponding to $J$ is
$
 L(J)=\left\{\,
 v-iJ(v)
 \;\middle|\;
 v\in V
 \,\right\}.
$
Conversely, every complex Lagrangian subspace has this form.
By \cite{Mein13}*{Thm.~1.4} the set $\mathcal{J}(V)$ is either empty or a homogeneous space, isomorphic to $\O(V)/\U(V,J).$ In particular, it has two path components.

The intersection of Lagrangian subspaces can be expressed in terms of the complex structures as $L(J_1)\cap L(J_2)=\Eig_{-1}(J_1J_2).$ From this it is easy to see that if $L\in\Lag_\C W$ and $J_1, J_2\in\cJ(V),$ then
\[
 (-1)^{\dim_\C L\cap L(J_1)}=\det(J_1J_2)(-1)^{\dim_\C L\cap L(J_2)}.
\]
As $\det(J_1)\in\{\pm1\}$ we see in particular that the left hand side only depends on the path-component of $J_1.$

%

\subsubsection*{Proof of bordism invariance}

Proposition~\ref{bio4prop4} implies that $C_\de(D_Y)$ is a complex Lagrangian subspace of $W_\de(D_X),$ so $V_{[0,\de]}(D_X)$ admits a complex structure for all $\de\ge 0.$ In particular, $\dim_\R V_{[0,\de]}(D_X)$ is even. Setting $\de=0,$ we find that $\dim_\R\Ker D_X$ is even, thus proving \eqref{bio3eq3}.\hfill\qedsymbol

\subsubsection*{Complex Lagrangians and orientations}

We already know that $\Lag_\C W_\de(D_X)$ is a nonempty space and thus has two path components. We now identify the $\Z_2$-torsor $\pi_0\Lag_\C W_\de(D_X)$ of path components with the $\Z_2$-torsor of orientations with the $\Z_2$-torsor of orientations $O(\PF_\R D_X)$ of the Pfaffian line.

The structure of \eqref{bio4eq20} and \eqref{bio4eq13} is the same, so the properties \eqref{bio4eq15}--\eqref{bio4eq16} carry over unchanged.
Every complex Lagrangian subspace $L_\de\subset W_\de(D_X)$ has the form $L_\de=L(J_\de)$ for a complex structure $J_\de\in\cJ(V_{[0,\de]}(D_X)).$ Up to deformation, there are two possibilities, distinguished by the induced orientation $\R_{>0}\om(J_\de)\in O(\det V_{[0,\de]}(D_X))$ as in \eqref{bioAeq7}. Inserting also a sign, this construction defines for all $\de\ge0$ an isomorphism
\e
\label{bio4eq21}
\begin{aligned}
 \pi_0\Lag_\C W_\de(D_X)&\longra O(\det V_{[0,\de]}(D_X)),\\
 [L_\de]=[L(J_\de)]&\longmapsto(-1)^{\dim_\C V_{[0,\de]}(D_X)}\R_{>0}\om(J_\de).
\end{aligned}
\e
If $0\le\de<\varepsilon,$ these fit into a commutative diagram (using notation from \S\ref{bioA3})
\[
\begin{tikzcd}
	\pi_0\Lag_\C W_\de(D_X)\arrow[dd,"\eqref{bio4eq12}"]\rar & O(\det V_{[0,\de]}(D_X))\arrow[rd,"\stab_{\de}^\skew(D_X)"]\arrow[dd,"\stab_{\varepsilon,\de}^\skew(D_X)"]\\
	&& O(\PF_\R D_X),\\
	\pi_0\Lag_\C W_\varepsilon(D_X)\rar & O(\det V_{[0,\varepsilon]}(D_X))\arrow[ru,"\stab_{\varepsilon}^\skew(D_X)"']
\end{tikzcd}
\]
which justifies the sign in \eqref{bio4eq21}. Hence there is an isomorphism
\e
\label{bio4eq22}
  \lim\limits_{\longra}\pi_0\Lag_\C W_\de(D_X)\overset{\smash\cong}{\longra} O(\PF_\R D_X).
\e

\subsubsection*{Construction of \textup{\eqref{bio3eq4}}}

Every complex Lagrangian subspace $L_\de\subset W_\de(D_X)$ determines a self-adjoint nearly APS boundary condition $B=B_\APS(-\de)\op L_\de$ for $D_Y.$ The mod-two index of a complex skew-adjoint Fredholm operator is locally constant and thus factors over the set of path components as
\[
 \pi_0\Lag_\C W_\de(D_X)\xrightarrow{\smash{i_{Y,\de}^\C}}\Z_2,
 \quad [L_\de]\longmapsto(-1)^{\dim_\C\Ker D_{Y,B_\APS(-\de)\op L_\de}}.
\]
(If $X=\emptyset,$ then $i_{Y,\de}^\C\colon\{\pm1\}\to\{\pm1\}$ is multiplication by the complex mod-two index $(-1)^{\dim_\C D_Y}.$) For $0\le\de<\varepsilon$ let $L_\varepsilon$ be the image of $L_\de$ under \eqref{bio4eq12}. From \eqref{bio4eq16} we find $B_\APS(-\de)\op L_\de=B_\APS(-\varepsilon)\op L_\varepsilon,$ so $i_{Y,\de}^\C[L_\de]=i_{Y,\varepsilon}^\C[L_\varepsilon].$ Therefore $i_{Y,\de}^\C$ factors through the inverse limit to an isomorphism
\e
\label{bio4eq23}
 \lim\limits_{\longra}\pi_0\Lag_\C W_\de(D_X)\overset{\smash\cong}{\longra}\{\pm1\}.
\e

\begin{dfn}
\label{bio4def6}
Let $\tau_\ell[D_Y]\colon O(\PF_\R D_X)\to\{\pm1\}$ be the composition of the inverse of \eqref{bio4eq22} with \eqref{bio4eq23}. In other words, for all $\de\ge 0$ we have
\[
 \tau_\ell[D_Y]\bigl(\R_{>0}\stab_{\de}^\skew(D_X)O(J_\de)\bigr)=(-1)^{\dim_\C\Ker D_{Y,B}+\dim_\C V_{[0,\de]}(D_X)},
\]
where $B=B_\APS(-\de)\op L(J_\de)$ and $J_\de$ is a complex structure on $V_{[0,\de]}(D_X).$ (If $X=\emptyset,$ then $O(\PF_\R D_X)=\{\pm1\}$ and $\tau_\ell[D_Y](\pm1)=\pm(-1)^{\dim_\C\Ker D_Y}.$)
\end{dfn}

Now let $\cD_X,$ $\cD_Y$ be families of adapted elliptic differential operators as in Theorem~\ref{bio3thm2}. Then \eqref{bio4eq22} and \eqref{bio4eq23} are continuous and hence $\tau_\ell[\cD_Y]=\coprod_{t\in T}\tau_\ell[D_Y(t)]$ is a continuous isomorphism of principal $\Z_2$-bundles over $T.$

Again, Theorem~\ref{bio3thm2}(a)--(e) can be verified point-wise. 

\subsubsection*{Proof of \textup{\ref{bio3thm2a}}}

For simplicity, assume $\6 Y=\emptyset$ again, so $\6 Y'=X_+\amalg X_-.$ Continue in the notation of \S\ref{bio411}. We have
\ea*
 \jmath_{X_-}^*(D_{Y'})
 &=(-\si)\left(\frac{\6}{\6 t}-A_{X_0}\right) &&\text{by \eqref{bio4eq8}}\\
 &=
 (-\si)\left(\frac{\6}{\6 t}-i D_{X_0}\right)
 && \text{by \eqref{bio4eq20}}\\
 &\cong
 \left(\frac{\6}{\6 t}-iD_{X_0}\right)\si
 && \text{conjugate by $i\si$}\\
 &=\si\left(\frac{\6}{\6 t}+iD_{X_0}\right) && \text{$\si$ anti-linear}\\
 &=\Cyl(D_{X_0}) && \text{by \eqref{bio2eq2}}
\ea*
and $\jmath_{X_-}^*(D_{Y'})=\Cyl(D_{X_0}).$ Hence $D_{X_\pm}\cong D_{X_0}$ and $D_{X_+\amalg X_-}\cong D_{X_0}\op D_{X_0}$ are isomorphic $\ell$-adapted operators. Let $\de\ge0.$ We can identify $V_{[0,\de]}(D_{X_0})\cong V_{[0,\de]}(D_{X_\pm})$ and $V_{[0,\de]}(D_{X_+}\op D_{X_-})\cong V_{[0,\de]}(D_{X_0})\op V_{[0,\de]}(D_{X_0}).$

Let $\om_\de\in \det V_{[0,\de]}(D_{X_0})$ be a volume form and $n=\dim_\R V_{[0,\de]}(D_{X_0}).$ The diagram \eqref{bio3eq5} involves the isomorphisms
\[
\begin{tikzcd}
	\begin{array}{l}
		\PF_\R(D_{X_+})\ot_\R\PF_\R(D_{X_-})\\
		\ni\om_\de\ot\om_\de
	\end{array}
	\dar{\eqref{bioAeq14}}\rar{\eqref{bioAeq13}} &
	\begin{array}{l}
	\PF_\R(D_{X_0})\ot_\R\PF_\R(D_{X_0})^*\\
	\ni (-1)^{n(n-1)/2}\om_\de\ot\om_\de^*
	\end{array}
	\dar{\an{\,,\,}}\\
	\PF_\R(D_{X_+}\op D_{X_-})\ni(\om_\de,0)\wedge(0,\om_\de) & \R\ni(-1)^{n(n-1)/2}.
\end{tikzcd}
\]

To evaluate the bottom horizontal map in \eqref{bio3eq5}, according to Definition~\ref{bio4def6}, we should write the orientation $\R_{>0}(\om_\de,0)\wedge(0,\om_\de)$ in the form $O(J_\de)$ for an orthogonal complex structure $J_\de$ on $V_{[0,\de]}(D_{X_+}\op D_{X_-}).$ Using the identification $V_{[0,\de]}(D_{X_+}\op D_{X_-})\cong V_{[0,\de]}(D_{X_0})\op V_{[0,\de]}(D_{X_0})$ define $J_\de(v,w)=(-w,v).$ By \eqref{bioAeq7} we then have $O(J_\de)=(-1)^{n(n-1)/2}(\om_\de,0)\wedge(0,\om_\de).$ Hence $\tau_\ell[D_{Y'}]$ from Definition~\ref{bio4def6} maps $(\om_\de,0)\wedge(0,\om_\de)$ to $(-1)^{\dim_\C\Ker D_{Y',B'}+n(n-1)/2}$ (note here that $\dim_\C V_{[0,\de]}(D_{X_+\amalg X_-})=2\dim_\C V_{[0,\de]}(D_{X_0})$ is always even), where $B'$ can be deformed to a transmission condition by \eqref{bio4eq9}, so $\dim_\C\Ker D_{Y',B'}=\dim_\C D_Y.$ Now put this calculation into the above diagram to complete the proof of (a).
\hfill\qedsymbol

\subsubsection*{Proofs of \textup{\ref{bio3thm2c}}, \textup{\ref{bio3thm2d}}, \textup{\ref{bio3thm2e}}}

The proofs of parts (c) and (d) are straight-forward consequences of the fact that by \eqref{bioAeq14} Pfaffian line bundles are compatible with direct sums and that the complex mod-two index is additive under direct sums.

Part (e) is obvious from Definition~\ref{bio4def6}.

\subsubsection{\texorpdfstring{Proof of Theorem~\textup{\ref{bio3thm2}} in the case $\ell\equiv 2$}{Proof of Theorem 3.2 in the case ℓ≡2}}
\label{bio433}

By \eqref{bio2eq4}, $E_0=\ol E_1,$ $F_0=F_1=E_0\op\ol{E_0},$ and over the collar
\[
 D_Y
 =
 \begin{pmatrix}
 	-i\frac{\6}{\6 t} & -\ol{D_X}\\
 	D_X & i\frac{\6}{\6 t}
 \end{pmatrix},
 \quad
 A_X
 =
 \begin{pmatrix}
 	0 & i\ol{D_X}\\
 	iD_X & 0
 \end{pmatrix}.
\]

Hence $\be_\de$ is a sesquilinear form and $W_\de(D_X)$ is a quaternionification.

\subsubsection*{Quaternionic Lagrangians}

Let $W$ be a \emph{quaternionification}, a quaternionic vector space of the form $W=\H\ot_\C V\cong V\op\ol V$ for a complex vector space $V.$ Let $\an{,}_V$ be a Hermitian metric on $V$ and define a sesquilinear form on $W$ by
$\be\bigl(
 (v_1,w_1),(v_2,w_2)
 \bigr)
 =
 i\bigl(\an{v_1,v_2}_V-\ol{\an{w_1,w_2}_V}\bigr)$ for $v_1, v_2\in V$ and $w_1,w_2\in\ol V.$
 
A quaternionic subspace $L\subset W$ is \emph{Lagrangian} if $L^{\perp_\be}=L.$ Let $\Lag_\H W$ be the set of quaternionic Lagrangians, where $\Lag_\H W=\{\pm1\}$ if $W=\{0\}.$ There is a bijection the space $\{J\in\Hom_\C(V,\ol V)\mid \ol{J}\circ J=-\id_V\}$ of orthogonal quaternionic structures on $V$ and $\Lag_\H W,$ where the Lagrangian associated to $J$ is $L=\Ga(J)\subset V\op\ol V.$ These sets are nonempty if and only if $\dim_\C V$ is even.

\subsubsection*{Proof of bordism invariance}

Put $\de=0.$ The space $W_\de(D_X)=\Ker A_X=\Ker D_X\op\ol{\Ker D_X}$ is a quaternionification of $V_\de(D_X)=\Ker D_X$ and is equipped with the sesquilinear form $\be_\de.$ By Proposition~\ref{bio4prop4}, $L_{Y,\de}$ is a Lagrangian subspace of $W_\de(D_X),$ which implies $\ind_\ell D_X=\dim_\C\Ker D_X\equiv 0\bmod{2}.$\hfill\qedsymbol
\medskip

All the remaining assertions in Theorem~\ref{bio3thm2} are empty.

\subsubsection{\texorpdfstring{Proof of Theorem~\textup{\ref{bio3thm2}} in the cases $\ell\equiv 3,7$}{Proof of Theorem 3.2 in the cases ℓ≡3,7}}
\label{bio434}

Let $\K=\H$ if $\ell\equiv 3$ and $\K=\R$ if $\ell\equiv 7.$ By \eqref{bio2eq5}, $E_0=E_1,$ $D_X\colon\Ga^\iy(E_0)\to\Ga^\iy(E_1)$ is $\K$-linear self-adjoint, $F_0|_X=F_1|_X=E_0,$ and over the collar
\[
 D_Y=\frac{\6}{\6 t} + D_X
\]

Since we always have $\ind_\ell D_X=0,$ we only need to define \eqref{bio3eq4}.

\begin{dfn}
\label{bio4def7}
Let $\tau_\ell[D_Y]\colon\SP_\K D_X\to \Z,$	$[m,\de]\mapsto m+\ind D_{Y,B_\APS(\de)},$ where $\de\in\R\setminus\Spec(D_X)$ and $m\in\Z.$ This map is well-defined by \eqref{bio4eq7}
\end{dfn}

Let $\cD_X,$ $\cD_Y$ be families of adapted elliptic differential operators as in Theorem~\ref{bio3thm2}. Since $\ind D_{Y,B_\APS(\de)}$ is continuous over $T_\de$ for all $\de\ge0,$ the map $\tau_\ell[\cD_Y]=\coprod_{t\in T}\tau_\ell[D_Y]$ is a continuous isomorphism of principal $\Z$-bundles.

Again, Theorem~\ref{bio3thm2}(a)--(e) can be verified point-wise.

\subsubsection*{Proof of \textup{\ref{bio3thm2a}}}

Suppose again for simplicity that $\6 Y=\emptyset$ and so $\6 Y'=X_+\amalg X_-.$ In view of \eqref{bio2eq5}, in \eqref{bio4eq8} we have $\jmath_{X_\pm}^*(D_Y)=\Cyl(D_{X_\pm})$ for $D_{X_\pm}=\pm D_{X_0}.$ Now \eqref{bio3eq5} involves
\[
\begin{tikzcd}
	\begin{array}{l}
		\SP_\K D_{X_+}\ot_\Z\SP_\K D_{X_-}\\
		\ni[m,\de]\ot[n,\de]
	\end{array}
\dar{\eqref{bioAeq15}}\rar{\eqref{bioAeq16}} &
\begin{array}{l}
\SP_\K D_{X_0}\ot_\Z(\SP_\K D_{X_0})^*\\
\ni[m,\de]\ot[n,-\de]^*
\end{array}
\dar{\an{\,,\,}+\ind D_Y}\\
	\SP_\K(D_{X_+}\op D_{X_-})
	\ni[m+n,\de]
 & \Z\ni m+n+\ind D_Y.
\end{tikzcd}
\]
Hence \eqref{bio4eq9} and the deformation invariance of the index of $\K$-linear Fredholm operators imply $m+n+\ind D_{Y',B_\APS(-\de)}=m+n+\ind D_Y.$\hfill\qedsymbol

\subsubsection*{Proofs of \textup{\ref{bio3thm2c}}, \textup{\ref{bio3thm2d}}, \textup{\ref{bio3thm2e}}}

The proofs of parts (c) and (d) are straight-forward consequences of the fact that by \eqref{bioAeq15} spectral covering is compatible with direct sums and that the index is additive under direct sums.

Part (e) is obvious from Definition~\ref{bio4def7}.

\subsubsection{\texorpdfstring{Proof of Theorem~\textup{\ref{bio3thm2}} in the case $\ell\equiv 4$}{Proof of Theorem 3.2 in the case ℓ≡4}}
\label{bio435}

By \eqref{bio2eq6}, $D_X\colon\Ga^\iy(E_0)\to\Ga^\iy(E_1)$ is an $\H$-linear differential operator, $F_0|_X=E_0\op E_1^\diamond,$ $F_1|_X=E_0^\diamond\op E_1,$ and
\[
 D_Y=
 \begin{pmatrix}
 i & 0\\
 0 & -i
 \end{pmatrix}
 \left[
 \frac{\6}{\6 t}
 +
 \begin{pmatrix}
 0&(iD_X)^*\\
 iD_X & 0
 \end{pmatrix}
 \right]
 =
 \begin{pmatrix}
 	i\frac{\6}{\6 t} & D_X^*\\
 	D_X & -i\frac{\6}{\6 t}
 \end{pmatrix}.
\]
In this case, $\Im(\be_0)=g_{\Ker D_X}-g_{\Ker D_X^*}$ is a symmetric $\R$-bilinear form as in \S\ref{bio431} of inertia index $\dim_\R\Ker D_X-\dim_\R\Ker D_X^*.$ Put $\de=0.$ By Proposition~\ref{bio4prop4} $C_\de(D_Y)$ is a Lagrangian subspace of $W_\de(D_X)$ and hence as in \S\ref{bio431}
\[
 \ind_\ell D_X=(\dim_\R\Ker D_X-\dim_\R\Ker D_X^*)/4=0.
\]

\subsubsection{\texorpdfstring{Proof of Theorem~\textup{\ref{bio3thm2}} in the cases $\ell\equiv 5,6$}{Proof of Theorem 3.2 in the cases ℓ≡5,6}}
\label{bio436}

All assertions are empty.

\appendix

\section{Determinant, Pfaffian, and spectral bundles}
\label{bioA}

This appendix recalls the construction of the determinant line bundle, Pfaffian line bundle, and spectral covering, and fixes several sign conventions for these.

\addsubsection{Sign convention}
\label{bioA1}

Already the construction of the determinant and Pfaffian line bundles requires a sign convention. We use the supersymmetric sign convention as in Upmeier~\cite{Upme}*{\S 3}, which determines the signs in various isomorphisms according to principle that passing two vectors past each other causes a sign.

Let $\K$ be a field. A \emph{graded line} is a pair $((-1)^n,L)$ where $n\in\Z_2$ and $L$ is a $1$-dimensional vector space over $\K.$ The \emph{determinant} of an $n$-dimensional vector space $V$ over $\K$ is the graded line $\det_\K V=((-1)^n,\La^n_\K V).$ The non-zero elements of $\det_\K V$ thus all have the degree $(-1)^n$ and are represented by \emph{volume forms} $\om=v_1\wedge\cdots\wedge v_n$ where $v_1,\ldots, v_n$ is an \emph{ordered basis} of $V$ over $\K.$ In this section, all tensor products, dimensions, bases, and exterior powers will be taken over the (skew) field $\K,$ which we sometimes drop from the notation.

 Let $L_1,$ $L_2$ be graded lines in degrees $(-1)^{n_1},$ $(-1)^{n_2}.$ The tensor product $L_1\ot L_2$ is put into degree $(-1)^{n_1+n_2}.$ If $V=U\op W$ is an ordered direct sum decomposition into $\K$-linear subspaces, there is an isomorphism
\e
\label{bioAeq1}
 \det U\ot \det W\longra \det V
\e
corresponding to the convention that bases $u_1,\ldots,u_m$ of $U$ and $w_1,\ldots,w_p$ of $W$ determine the ordered basis $u_1,\ldots,u_m,w_1,\ldots,w_p$ of $V.$

The convention for the braid isomorphism is
\e
\label{bioAeq2}
 L_1 \ot L_2\overset{\smash\cong}{\longra}L_2\ot L_1,
\qquad
 \ell_1\ot \ell_2 \longmapsto (-1)^{n_1n_2} \ell_2 \ot \ell_1.
\e

Write $V^*=\Hom_\K(V,\K)$ for the dual vector space over $\K.$ Functionals are always evaluated on the left, $L\ot L^*\to\K,$ $\ell\ot\la\mapsto\an{\ell,\la},$ while evaluation on the right introduces a sign $(-1)^n.$ This implies the convention
\e
\label{bioAeq3}
	L_1^*\ot L_2^* \overset{\cong}{\longra} (L_2\ot L_1)^*,
	\qquad
	\an{x_2\ot x_1,\al_1\ot \al_2}\coloneqq
	\langle x_1, \al_1\rangle \langle x_2, \al_2\rangle.
\e
In the same way, the evaluation isomorphism $\an{\,,}\colon\det V\ot\det(V^*)\to\K$ is
\e
\label{bioAeq4}
\an{v_1\wedge\ldots\wedge v_n,\al_n\wedge\ldots\wedge\al_1}=\det[\an{v_j,\al_i}]_{i,j=1}^n,
\e
which differs by $(-1)^{n(n-1)/2}$ from the na\"ive convention.

We order the dual basis of $V^*$ as $v_n^*,\ldots,v_1^*,$ where $\an{v_i,v_j^*}=\de_{i,j},$ so the dual volume form in $\det(V^*)$ is given in accordance with \eqref{bioAeq4} by
\e
\label{bioAeq5}
 \om^*=v_n^*\wedge\cdots\wedge v_1^*.
\e

The determinant of a $\K$-linear isomorphism $f\colon V\to W$ is the element
\e
 \det(f)=v_1\wedge\cdots\wedge v_n\ot f(v_n)^*\wedge\cdots\wedge f(v_1)^*
 \in \det V\ot\det(W^*),
\label{bioAeq6}
\e
which is independent of the choice of basis $v_1,\ldots, v_n$ of $V.$

Let $\K=\R$ or $\C$ and let $J\colon V\to\ol V$ be an orthogonal, skew-adjoint $\K$-linear isomorphism satisfying $\ol J\circ J=-\id_V.$ Hence $J$ is a \emph{complex structure} on $V$ if $\K=\R$ and a \emph{quaternionic structure} on $V$ if $\K=\C.$ Moreover, $J$ induces a canonical volume form
\e
\label{bioAeq7}
\om(J)=v_1\wedge\cdots\wedge v_m\wedge J(v_m)\wedge\cdots\wedge J(v_1)\in\det V,
\e
where $v_1,\ldots, v_m$ is an arbitrary basis of $V$ over $\K.$

\addsubsection{Determinant line bundles}
\label{bioA2}

For Fredholm operators there is an interesting orientation problem for the `virtual difference' of the kernel and cokernel vector spaces, defined as follows.

If $D\colon H_1\to H_2$ is a Fredholm operator of Hilbert spaces over $\K=\R,$ or $\C,$ then $D^*D$ and $DD^*$ are self-adjoint and have real spectra which are discrete near zero with finite-dimensional eigenspaces. For simplicity, suppose that the entire spectrum is discrete (this holds for all applications in this paper; in general, one restricts to the `small' spectrum as in \cite{Upme}). Define
\ea
 \label{bioAeq8}
 V_\la^+(D)&=\Eig_{\la^2}(D^*D),
&V_\la^-(D)&=\Eig_{\la^2}(DD^*)
\ea
for $\la\ge 0$ and set $V_I^\pm(D)=\bigoplus_{\la\in I}V_\la^\pm(D)$ for $I\subset[0,\iy).$ These are finite-dimensional vector spaces and there is are orthogonal isomorphisms
\e
 \label{bioAeq9}
 J_\la(D)=\la^{-1}D|_{V_\la^+},\colon V_\la^+(D)\longra V_\la^-(D)
\e
for $\la\neq0$ and $J_I(D)=\bigoplus_{\la\in I} J_\la(D)\colon V_I^+(D)\to V_I^-(D)$ if $0\notin I.$

\begin{dfn}
\label{bioAdef1}
Let $D\colon H_1\to H_2$ be a $\K$-linear Fredholm operator of Hilbert spaces over $K=\R$ or $\C.$ For all $0\le\de<\varepsilon$ there is a direct sum decomposition $V_{[0,\varepsilon]}^\pm(D)=V_{[0,\de]}^\pm(D)\op V_{(\de,\varepsilon]}^\pm(D).$ The determinant $\det J_{(\de,\varepsilon]}(D)\in\det V_{(\de,\varepsilon]}^+(D)\ot\det V_{(\de,\varepsilon]}^-(D)^*$ defines an isomorphism $\stab_{\de,\varepsilon}$ by
\[
\begin{tikzcd}[column sep=huge]
\begin{array}{l}
\det V_{[0,\de]}^+(D)\\
\ot\det V_{[0,\de]}^-(D)^*
\end{array}
\rar{1\ot\det J_{(\de,\varepsilon]}(D)\ot 1}
\dar[dashed]{\stab_{\de,\varepsilon}(D)}
&
\begin{array}{l}
\det V_{[0,\de]}^+(D)\ot\det V_{(\de,\varepsilon]}^+(D)\\
\ot\det V_{(\de,\varepsilon]}^-(D)^*\ot\det V_{[0,\de]}^-(D)^*
\end{array}
\dar{\eqref{bioAeq3}}
\\
\det V_{[0,\varepsilon]}^+(D)\ot\det V_{[0,\varepsilon]}^-(D)^*
&
\begin{array}{l}
\det V_{[0,\de]}^+(D)\ot\det V_{(\de,\varepsilon]}^+(D)\\
\ot(\det V_{[0,\de]}^-(D)\ot\det V_{(\de,\varepsilon]}^-(D)^*)^*.
\end{array}
\lar{\eqref{bioAeq1}\ot\eqref{bioAeq1}^{*,-1}}
\end{tikzcd}
\]
Hence if $v_1^\pm,\ldots,v_{n_\pm}^\pm$ are ordered bases of $V_{[0,\de]}^\pm(D)$ and $w_1,\ldots,w_m$ is a basis of $V_{(\de,\varepsilon]}^+(D),$ then $\stab_{\de,\varepsilon}(D)(v_1^+\wedge\cdots\wedge v_{n_+}^+\ot (v_{n_-}^-)^*\wedge\cdots\wedge (v_1^-)^*)=v_1^+\wedge\cdots\wedge v_{n_+}^+\wedge w_1\wedge\cdots\wedge w_m\ot J_{(\de,\varepsilon]}(D)(w_m)^*\wedge\cdots\wedge J_{(\de,\varepsilon]}(D)(w_1)^*\wedge (v_{n_-}^-)^*\wedge\cdots\wedge (v_1^-)^*.$

The maps $\stab_{\de,\varepsilon}(D)$ form a direct system and the \emph{determinant line} of $D$ is
\[
 \DET_\K D=\lim\limits_{\longra}\det_\K V_{[0,\de]}^+(D)\ot\det_\K V_{[0,\de]}^-(D)^*.
\]
Let $\stab_\de(D)\colon\det_\K V_{[0,\de]}^+(D)\ab\ot\det_\K V_{[0,\de]}^-(D)^*\to\DET_\K D$ be the canonical map.
\end{dfn}


A \emph{continuous $T$-family} of $\K$-linear Fredholm operators is a continuous map $\cD\colon T\to B_\K(H_1,H_2),$ $t\mapsto D(t),$ into the Banach space of bounded operators with the operator norm such that $D(t)$ is a Fredholm operator for all $t\in T.$ Define open subsets $T_\de=\{t\in T\mid\de^2\notin\Spec(D(t)^*D(t)\}$ of $T$ for all $\de\ge0.$

\begin{dfn}
\label{bioAdef2}
Let $\cD=\{D(t)\}$ be a continuous $T$-family of $\K$-linear Fredholm operators. The \emph{determinant line bundle} is the disjoint union $$\DET_\K\cD=\coprod\nolimits_{t\in T}\DET_\K D(t)$$ with the unique topology such that all $\stab_\de(\cD)|_{T_\de}=\coprod_{t\in T_\de}\stab_{\de}(D(t))$ are isomorphisms of line bundles over $T_\de$ (the domain is naturally a line bundle).
\end{dfn}

Our sign convention determines various isomorphisms. For the adjoint Fredholm operator $D^*\colon H_2\to H_1$ we have $V_{[0,\de]}^\pm(D^*)=V_{[0,\de]}^\mp(D)$ and there is a unique isomorphism such that
\begin{equation}
\begin{tikzcd}
	\DET(D^*)\rar[dashed]&\DET(D)^*\dar{\stab_{\de}(D)^*}\\
	\det V_{[0,\de]}^+(D^*)\ot \det V_{[0,\de]}^-(D^*)^*
	\uar{\stab_{\de}(D^*)}\rar{\eqref{bioAeq3}}
	&
\mathrlap{\bigl(\det V_{[0,\de]}^+(D)\ot \det V_{[0,\de]}^-(D)^*\bigr)^*}
\hspace{25ex}
\end{tikzcd}
\label{bioAeq10}
\end{equation}
commutes for all $\de\ge0$ (using the usual identification of a vector space with its double dual). For a family a Fredholm operators $\cD$ the lower horizontal map depends continuously on $t\in T_\de$ and the vertical maps are homeomorphisms there, so the resulting isomorphism $\DET(\cD^*)\to\DET(\cD)^*$ is continuous.

For the negative Fredholm operator we have $V_{[0,\de]}^\pm(-D)=V_{[0,\de]}^\pm(D)$ and there is a unique isomorphism such that
\begin{equation}
\begin{tikzcd}[column sep=15ex]
	\DET(-D)\rar[dashed] &\DET D\\
	\det V_{[0,\de]}^+(-D)\ot\det V_{[0,\de]}^-(-D)^*
\rar{(-1)^{\dim V_{[0,\de]}^-(D)}}
\uar{\stab_{\de}(-D)} 
&
\mathrlap{\det V_{[0,\de]}^+(D)\ot\det V_{[0,\de]}^-(D)^*}
\hspace{20ex}
	\uar{\stab_{\de}(D)}
\end{tikzcd}
\label{bioAeq11}
\end{equation}
commutes for all $\de\ge 0$ (the sign is needed because the vertical isomorphisms are different). The isomorphism \eqref{bioAeq11} is continuous in families.

For the direct sum of Fredholm operators $D,$ $\tilde D$ the eigenspace decomposes as $V_{[0,\de]}^\pm(D\op\tilde D)=V_{[0,\de]}^\pm(D)\op V_{[0,\de]}^\pm(\tilde D)$ and there is a unique isomorphism such that
\begin{equation}
\begin{tikzcd}[row sep=tiny,column sep=tiny]
	\DET D\ot\DET\tilde D\arrow[rr,dashed]
	& &
	\DET(D\op\tilde D)
	\\[3ex]
	\begin{array}{l}
	\det V_{[0,\de]}^+(D)
	\ot\det V_{[0,\de]}^-(D)^*\\
	\ot\det V_{[0,\de]}^+(\tilde D)
	\ot\det V_{[0,\de]}^-(\tilde D)^*
	\end{array}
	\uar{\stab_{\de}(D)\ot\stab_{\de}(\tilde D)}
	\arrow[dr,shorten >=5pt,"{\eqref{bioAeq2}}"]
	& &
	\begin{array}{l}
	\det V_{[0,\de]}^+(D\op \tilde D)\\
	\ot\det V_{[0,\de]}^-(D\op \tilde D)^*
	\end{array}
	\uar{\stab_{\de}(D\op\tilde D)}
	\\
	&
	\mathclap{\begin{array}{l}
	\det V_{[0,\de]}^+(D)\ot\det V_{[0,\de]}^+(\tilde D)\\
	\ot
	\det V_{[0,\de]}^-(\tilde D)^*\ot \det V_{[0,\de]}^-(D)^*
	\end{array}}
	\hspace{15ex}
	\arrow[ru,shorten <=-10pt,"{\eqref{bioAeq1},\eqref{bioAeq3}}"']
	&
\end{tikzcd}
\label{bioAeq12}
\end{equation}
commutes for all $\de\ge 0.$ This differs from na\"ive rearrangement by the sign
\[
(-1)^{\dim V_{[0,\de]}^-(D)(\dim V_{[0,\de]}^+(D)+\dim V_{[0,\de]}^+(\tilde D))}.
\]
Moreover, \eqref{bioAeq12} is continuous in families.

\addsubsection{Pfaffian line bundles}
\label{bioA3}

Let $\K=\R$ or $\C.$ A $\K$-linear Fredholm operator $D\colon H\to \ol H$ is \emph{skew-adjoint} if $D^*=-\ol D.$ For a continuous family of skew-adjoint Fredholm operators, the principal bundle $O(\DET_\R \cD)\to T$ admits a continuous global section by  \cite{Upme}*{\S 3.3}, so the orientation problem is uninteresting. Instead, there is now an interesting orientation problem for the kernel alone.

Write $V_\la(D)=V_\la^+(D)=\Eig_{\la^2}(D^*D)$ in \eqref{bioAeq8} so that $\ol{V_\la(D)}=V_\la^-(D)=\Eig_{\la^2}(DD^*),$ and let $V_I(D)=\bigoplus_{\la\in I} V_\la(D)$ for $I\subset[0,\iy).$ The isomorphism $J_\la(D)\colon V_\la(D)\to \ol{V_\la(D)}$ in \eqref{bioAeq9} is now orthogonal and skew-adjoint and hence determines by \eqref{bioAeq7} a volume form $\om(J_I(D))\in\det V_I(D)$ for all $I\subset(0,\iy).$

\begin{dfn}
\label{bioAdef3}
Let $D\colon H\to \ol H$ be a $\K$-linear skew-adjoint Fredholm operator. For all $0\le\de<\varepsilon$ there is a direct sum decomposition $V_{[0,\varepsilon]}(D)=V_{[0,\de]}(D)\op V_{(\de,\varepsilon]}(D)$ and hence isomorphisms
\[
 \stab_{\de,\varepsilon}^\skew(D)\colon
 \det V_{[0,\de]}(D)\longra\det V_{[0,\varepsilon]}(D),\quad
 v\longmapsto v\wedge\om(J_{(\de,\varepsilon]}(D))
\]
The maps $\stab_{\de,\varepsilon}^\skew(D)$ form a direct system, and the \emph{Pfaffian line} of $D$ is
\[
 \PF_\K D=\lim\limits_{\longra}\det_\K V_{[0,\de]}(D).
\]
Let $\stab_\de^\skew(D)\colon\det_\K V_{[0,\de]}(D)\to\PF_\K D$ be the canonical map.
\end{dfn}

\begin{dfn}
\label{bioAdef4}
Let $\cD=\{D(t)\}$ be a continuous $T$-family of $\K$-linear skew-adjoint Fredholm operators. The \emph{Pfaffian line bundle} is
\[
 \PF_\K\cD=\coprod\nolimits_{t\in T}\PF_\K D(t)
\]
with the unique topology such that all $\stab_\de^\skew(\cD)|_{T_\de}=\coprod_{t\in T_\de}\stab_\de^\skew(D(t))$ are isomorphisms of line bundles over $T_\de$ (the domain is naturally a line bundle).
\end{dfn}

There is a unique isomorphism $\PF(D^*)\to\PF(D)^*$ such that
\begin{equation}
\begin{tikzcd}
 \PF(D^*)\arrow[r,dashed] & \PF(D)^*\dar{\stab_\de^\skew(D)^*}\\
 \det V_{[0,\de]}(D^*)\rar
 =\det \ol{V_{[0,\de]}(D)}
 \uar{\stab_\de^\skew(D^*)}
 & \det V_{[0,\de]}(D)^*
\end{tikzcd}
\label{bioAeq13}
\end{equation}
where the bottom horizontal map is induced by the isomorphism $\ol{V_{[0,\de]}(D)}\to V_{[0,\de]}(D)^*$ from the inner product on $V_{[0,\de]}(D).$ Hence if $v_1,\ldots, v_n$ is an orthonormal basis of $V_{[0,\de]}(D)$ over $\K,$ then the lower horizontal map in \eqref{bioAeq13} is $v_1\wedge\cdots\wedge v_n\mapsto v_1^\flat\wedge\cdots\wedge v_n^\flat=v_1^*\wedge\cdots\wedge v_n^*.$ \eqref{bioAeq13} is continuous in families.

Let $D,$ $\tilde D$ be skew-adjoint Fredholm operators. For the direct sum there is a unique isomorphism $\PF D\ot\PF \tilde D\to\PF(D\op\tilde D)$ such that
\begin{equation}
\begin{tikzcd}
	\PF D\ot\PF\tilde D\arrow[r,dashed] & \PF(D\op\tilde D)\\
	\det V_{[0,\de]}(D)\ot\det V_{[0,\de]}(\tilde D)\uar{\stab_\de^\skew(D)\ot\stab_\de^\skew(\tilde D)}\rar
	&
	\det V_{[0,\de]}(D\op\tilde D)\uar{\stab_\de^\skew(D\op\tilde D)}
\end{tikzcd}	
\label{bioAeq14}
\end{equation}
commutes for all $\de\ge0,$ where the lower horizontal map is given by \eqref{bioAeq1} and the direct sum decomposition $V_{[0,\de]}(D\op\tilde D)=V_{[0,\de]}(D)\op V_{[0,\de]}(\tilde D).$ The isomorphism is continuous in families.

\addsubsection{Spectral coverings}
\label{bioA4}

In contrast to the skew-adjoint case, $O(\DET_\R\cD)$ may be non-trivial even when every $D(t)$ is self-adjoint (for an example, see \cite{JoUp1}*{\S 2.4} and \cite{JTU}). We can now define the harder problem of `enumerating' the spectrum, which is a generalization of the orientation problem (see Proposition~\ref{bioAprop1} below).

\begin{dfn}
\label{bioAdef5}
Let $\K=\R,\C,$ or $\H.$ A $\K$-linear Fredholm operator $D\colon H\to H$ is \emph{self-adjoint} if $D^*=D.$ The \emph{spectral $\Z$-torsor} $\SP_\K D$ is the set of equivalence classes of pairs $(m,\de)\in\Z\t(\R\setminus\Spec(D))$ where $(m,\de)$ is equivalent to $(n,\varepsilon)$ for $\de<\varepsilon$ if $n-m=\dim_\K\Eig_{[\de,\varepsilon]}(D).$ For $\de\in\R\setminus\Spec(D)$ let $\stab_\de^\sym\colon\Z\to\SP_\K D$ be the quotient projection $m\mapsto[(m,\de)].$

Let $\cD=\{D(t)\}$ be a continuous $T$-family of $\K$-linear self-adjoint Fredholm operators. The \emph{spectral covering} of $\cD$ is the principal $\Z$-bundle
\[
 \SP_\K\cD=\coprod\nolimits_{t\in T}\SP_\K D(t)
\]
over $T$ with the unique topology for which all $\coprod_{t\in T_\de}\stab_\de^\sym\colon T_\de\t\Z\to\SP_\K\cD|_{T_\de}$ are isomorphisms of principal $\Z$-bundles over $T_\de.$
\end{dfn}

\begin{rem}
\label{bioArem1}
If the spectrum of $D$ is bi-infinite (for example, for a first order elliptic differential operator), the spectral torsor $\SP_\K D$ can be identified with the $\Z$-torsor of enumerations $\cdots\le \la_{-1}\le\la_0\le\la_1\le\cdots$ of the spectrum of $D$ (eigenvalues repeated according to multiplicity). The $\Z$-action shifts the labels.
\end{rem}

For all $k\in\N$ we can quotient $\SP_\K\cD$ by the action of the subgroup $k\Z\subset\Z$ and obtain a principal $\Z_k$-bundle $\SP_\K\cD/k\Z,$ whose elements are called \emph{mod $k$ symmetric orientations}. This principal bundle is relevant to the idea of \emph{Floer gradings} as in Donaldson~\cite{Do}*{\S 3.3}, and the arguments there show that for a compact oriented Riemannian $4$-manifold $X$ and principal $\SU(2)$-bundle $P\to X$ the principal $\Z_8$-bundle $\SP_\R(\cD)/8\Z$ can be $\Aut(P)$-equivariantly trivialized over the space of connections $\cA_P,$ where $\cD$ is the $\cA_P$-family of twisted ASD-operators $(\d_+,\d^*)^{\na_{\Ad P}}$ on $X.$

\begin{prop}
\label{bioAprop1}
If\/ $D$ is a self-adjoint\/ $\R$-linear Fredholm operator, then there is a canonical isomorphism of principal\/ $\Z_2$-bundles\/ $\SP_\R(D)/2\Z\cong O(\DET_\R D),$ which is continuous in families. In particular, mod\/ $k$ symmetric orientations with even\/ $k\in\N$ determine ordinary orientations.
\end{prop}

\begin{proof}
Pick $\la\in\R\setminus\Spec(D)$ and define a family $\cD$ of $\R$-linear Fredholm operators by $D(t)=D-t\la$ for $t\in[0,1].$ Then $O(\DET_\R \cD)\to[0,1]$ is a principal $\Z_2$-bundle. Let $\de=0.$ Note that $V_{[0,\de]}^\pm(D-\la)=\{0\}$ since $D-\la$ is invertible and hence $\stab_\de(D-\la)\colon\R\to\DET_\R(D-\la)$ is an isomorphism and $O(\DET_\R(D-\la)).$ In particular, we have a canonical orientation $\R_{>0}\stab_\de(1)$ and we let $o_\la(D)\in O(\DET_\R D)$ be its image under the fiber transport isomorphism $O(\DET_\R(D-\la))\to O(\DET_\R D)$ in the principal $\Z_2$-bundle between the endpoints of the interval $[0,1].$ Now define $\SP_\R(D)/2\Z\to O(\DET_\R D),$ $[m+2\Z,\la]\mapsto (-1)^m o_\la(D).$ It is not hard to check that this map is independent of the choice of $\la.$ This also implies that the map is continuous in families.
\end{proof}

For the direct sum of self-adjoint Fredholm operators there is an isomorphism
\e
 \SP_\K D\ot_\Z\SP_\K \tilde D\longra\SP_\K(D\op\tilde D)
\label{bioAeq15}
\e
that maps $[m,\de]\ot[\tilde m,\de]\mapsto [m+\tilde m,\de]$ for all $\de\in\R\setminus(\Spec(D)\cup\Spec(\tilde D))$ and $m,\tilde m\in\Z.$ The isomorphism is continuous in families.

For the negative there is an isomorphism
\e
 \SP_\K D\longra\SP_\K(-D)^*=\Hom_\Z(\SP_\K(-D),\Z)
\label{bioAeq16}
\e
mapping $[m,\de]$ to the morphism $[m,\de]^*\colon[n,-\de]\mapsto m+n$ for all $\de\in\R\setminus\Spec(D).$

\section{Real Dirac operators}
\label{bioB}

The objective this appendix is to explain the precise connection between the real Dirac operator on a manifold with boundary and the real Dirac operator on its boundary, for which the author could not find a good reference.

\addsubsection{Clifford algebras}
\label{bioB1}

We first recall the basic terminology for Clifford algebras and spinor representations. Lawson--Michelsohn~\cite{LaMi} and Atiyah--Bott--Shapiro~\cite{AtBoSh} are good general references.

\begin{dfn}
\label{bioBdef1}
The real $n$-dimensional \emph{Clifford algebra} is the associative, unital $\R$-algebra $\Cl_n$ generated by $e_1,\ldots, e_n$ subject to the \emph{Clifford relations}
\[
 e_k e_\ell+e_\ell e_k=-2\de_{k,\ell}\qquad (1\le k,\ell\le n).
\]

There is an algebra automorphism $\al\colon\Cl_n\to\Cl_n$ of the Clifford algebra defined on all generators by $\al(e_k)=-e_k.$ Since $\al$ is an involution, it splits the Clifford algebra into its $(-1)^j$-eigenspaces $\Cl_n^j$ for $j=0,1,$ the \emph{even part} $\Cl_n^0$ and the \emph{odd part} $\Cl_n^1.$ Then $\Cl_n^0\subset\Cl_n$ is a subalgebra. Since $e_ke_n$ for $1\le k\le n-1$ satisfy the Clifford relations for $\Cl_{n-1}$ there is an isomorphism $j_{n-1}\colon\Cl_{n-1}\to\Cl_n^0$ of ungraded algebras satisfying $j_{n-1}(e_k)=e_ke_n.$

The \emph{volume element} in the Clifford algebra $\Cl_n$ is $\om_n=e_1\cdots e_n.$ Using the Clifford relations one finds (see \cite{LaMi}*{Prop.~3.3})
\ea
\label{bioBeq1}
\om_n^2&=(-1)^{n(n+1)/2},
&\om_n e_k&=(-1)^{n+1}e_k\om_n\qquad(1\le k\le n).
\ea
\end{dfn}

Here is the classification of real Clifford algebras.

\begin{thm}[see {\cite{LaMi}*{Thm.~I.4.3}}]
\label{bioBthm1}
For\/ $n\not\equiv 3,7$ the Clifford algebra\/ $\Cl_n$ is isomorphic to a matrix algebra over\/ $\R,$ $\C,$ or\/ $\H$ and for\/ $n\equiv 3,7$ is isomorphic to a direct sum of two copies of the same matrix algebra.
\end{thm}

\begin{table}[htb]
\centerline{\begin{tabular}{|l|c|c|c|c|c|c|c|c|c|c|c|}
\hline
$n$ & $1$ & $2$ & $3$ & $4$ & $5$ & $6$ & $7$ & $8$\\
\hline
$\Cl_n$ & $\C$ & $\H$ & $\H\op\H$ & $\H^{2\t 2}$ & $\C^{4\t 4}$ & $\R^{8\t 8}$ &
 $\begin{array}{c}
  \R^{8\t 8}\\\op\R^{8\t 8}
 \end{array}$
& $\R^{16\t 16}$\\
\hline
\end{tabular}}
\caption{Examples of Clifford algebras $\Cl_n$ for $1\le n\le 8,$ extended upwards periodically by $\Cl_{n+8}\cong\Cl_n\ot_\R\R^{8\t 8}.$}
\label{bioBtab1}
\end{table}

According to Theorem~\ref{bioBthm1} the Clifford algebra $\Cl_n$ has a unique irreducible algebra representation $\rho_{\De_n}\colon\Cl_n\to\End_\R(\De_n)$ for $n\not\equiv 3,7$ and a pair of inequivalent irreducible algebra representations $\rho_{\De_n^\pm}\colon\Cl_n\to\End_\R(\De_n^\pm)$ for $n\equiv 3,7,$ and we fix invariant metrics on these representations. From Table~\ref{bioBtab1} we can read off for which $n$ these representations are real, complex, or quaternionic. As we are interested in the \emph{real} Dirac operator here, all our representations will be treated as \emph{real} representations, and we will view a complex structure $I$ or a quaternionic structure $I,J$ on a representation as extra data.

For an algebra representation $\rho\colon B\to\End_\R(V)$ and morphism $f\colon A\to B$ of algebras, write $f^*(V)$ for the vector space $V$ with the pullback representation $\rho\circ f.$ For example, on $\al^*(\De_n)$ the element $e_k$ acts by $\rho_{\al^*(\De_n)}(e_k)=-\rho_{\De_n}(e_k).$ When $f$ is an inclusion we also write $V|_A$ for $f^*(V).$ For example, the inclusion $i_{n-1}\colon\Cl_{n-1}\to\Cl_n$ is defined on generators by $i_{n-1}(e_k)=e_k$ for $1\le j\le n-1.$

If $n\equiv 3,7,$ then \eqref{bioBeq1} implies that $\om_n\in\Cl_n$ is central, satisfies $\om_n^2=1,$ and thus acts by scalars on the irreducible representations $\De_n^\pm.$ By passing to $\al^*(\De_n^\pm)$ if necessary, we may always arrange for $\rho_{\De_n^\pm}(\om_n)=\pm1.$

\addsubsection{Positive and skew-adjoint Dirac operators}
\label{bioB2}

We now define the real Dirac operator, which refers to the positive Dirac operator if $n\equiv 0,4$ and the skew-adjoint Dirac operator if $n\equiv 1,2.$

Recall from \cite{LaMi}*{Ch.~I, \S 2} that the connected double cover $\Spin(n)\to\SO(n)$ is a subgroup of the group of units of the even part of the Clifford algebra. The \emph{spinor representation} $\Si_n=\De_n^{(+)}|_{\Spin(n)}$ is the restriction of the irreducible Clifford algebra representation to $\Spin(n)\subset\Cl_n.$
Notice here our convention that $\De_n$ is an algebra representation and $\Si_n$ is a Lie group representation. The spinor representation is equipped with a \emph{Clifford multiplication} $\Cl_n\ot\Si_n\to\Si_n,$ $e_k\ot\psi\mapsto\rho_{\De_n^{(+)}}(e_k)(\psi).$

\begin{dfn}
\label{bioBdef2}
A  Riemannian spin manifold $X,$ possibly with boundary or corners, has a principal $\Spin(n)$-bundle $P_\Spin(X)\to X.$ The \emph{real spinor bundle} is the vector bundle $\Si_X=(P_\Spin(X)\t\Si_n)/\Spin(n)$ associated to the spinor representation. There is a natural metric and connection on $\Si_X$ induced by the Levi-Civita connection on $TX.$ The (full) \emph{real Dirac operator} $D\colon\Ga^\iy(\Si_X)\to\Ga^\iy(\Si_X)$ is defined in a local orthonormal frame by $
D=\sum_{k=1}^n \rho_{\De_n^{(+)}}(e_k)\na_{e_k}.$
\end{dfn}

On a Riemannian spin manifold the Dirac operator is a globally defined, formally self-adjoint, first-order elliptic differential operator \cite{LaMi}*{Ch.~II, \S 5}.
\medskip

Sometimes we can `halve' the real Dirac operator: Since $\Cl_n^0\cong\Cl_{n-1},$ $\Spin(n)\subset\Cl_n^0,$ and $\dim_\R\De_n=2\dim_\R\De_{n-1}^{(+)}$ for $n\equiv 0,1,2,4$ by Table~\ref{bioBtab1}, the real spinor representation $\Si_n$ becomes reducible for $n\equiv 0,1,2,4$ and splits into \emph{half-spinor representations} $\Si_n^\pm,$ exchanged by the Clifford multiplication $\rho_{\De_n}(e_k)\colon\Si_n^\pm\to\Si_n^\mp.$

\begin{itemize}
\item
If $n\equiv 0,$ the real representation $\Si_n$ decomposes into non-isomorphic real subrepresentations $\Si_n^\pm,$ which are the $(\pm 1)$-eigenspaces of $\rho_{\De_n}(\om_n).$
\item
If $n\equiv 1,$ the complex representation $(\Si_n,I)$ admits a real structure $\si\colon\ab\Si_n\to\ol{\Si_n}$ and decomposes into isomorphic real subrepresentations with $\Si_n^+=I(\Si_n^-)\subset\Si_n,$ where $\Si_n^\pm$ are the $(\pm1)$-eigenspaces of $\si.$ Moreover, the complex structure coincides with the volume form $I=\rho_{\De_n}(\om_n).$
\item
If $n\equiv 2,$ the quaternionic representation $(\Si_n,I,J)$ decomposes into a complex conjugate pair of complex subrepresentations with $\Si_n^+=J(\Si_n^-)\subset\Si_n,$ where $\Si_n^\pm$ are the $(\pm i)$-eigenspaces of $\rho_{\De_n}(\om_n).$ 
\item
If $n\equiv 4,$ the quaternionic representation $(\Si_n,I,J)$ decomposes into non-isomorphic quaternionic subrepresentations $\Si_n^\pm,$ the $(\pm 1)$-eigenspaces of $\rho_{\De_n}(\om_n).$
\end{itemize}

\begin{dfn}
\label{bioBdef3}
Let $\dim X\equiv 0,1,2,4.$ The decomposition $\Si_n=\Si_n^+\op\Si_n^-$ induces a decomposition $\Si_X=\Si_X^+\op \Si_X^-$ into \emph{positive} and \emph{negative real spinor bundles} and $\slashed{D}_X$ can be decomposed into \emph{positive} and \emph{negative Dirac operators}
\[
 \slashed{D}_X=
 \begin{pmatrix}
 	0 & \slashed{D}_X^-\\
 	\slashed{D}_X^+ & 0
 \end{pmatrix},
 \quad
 \slashed{D}_X^\pm\colon\Ga^\iy\bigl(\Si_X^\pm\bigr)\longra\Ga^\iy\bigl(\Si_X^\mp\bigr),
 \quad \bigl(\slashed{D}_X^+\bigr)^*=\slashed{D}_X^-.
\]
When $n\equiv 1,2$ we can use the complex and quaternionic structure to define from $\slashed{D}_X^+$ the \emph{skew-adjoint Dirac operator} $\slashed{D}_X^\skew$ as the composition
\ea
 &\Ga^\iy(\Si_X^+)\xrightarrow{\smash{\slashed{D}_X^+}}\Ga^\iy(\Si_X^-)\overset{\smash{-I}}{\longra}\Ga^\iy(\Si_X^+) &&\text{if $n\equiv 1,$}\label{bioBeq2}\\
 &\Ga^\iy(\Si_X^+)\xrightarrow{\smash{\slashed{D}_X^+}}\Ga^\iy(\Si_X^-)\overset{\smash{-J}}{\longra}\Ga^\iy(\ol{\Si_X^+}) &&\text{if $n\equiv 2.$}
 \label{bioBeq3}
\ea
If $n\equiv 1,$ then $\slashed{D}_X^\skew$ is $\R$-linear and $\bigl(\slashed{D}_X^\skew\bigr)^*=-\slashed{D}_X^\skew.$ If $n\equiv 2,$ then $\slashed{D}_X^\skew$ is $\C$-linear and $\bigl(\slashed{D}_X^\skew\bigr)^*=-\ol{\slashed{D}_X^\skew}.$ When we speak of `the' \emph{real Dirac operator} $D_X$ in dimension $n,$ we mean the positive Dirac operator $\slashed{D}_X^+$ if $n\equiv 0,4,$ the skew-adjoint Dirac operator $\slashed{D}_X^\skew$ if $n\equiv 1,2,$ and the full Dirac operator $\slashed{D}_X$ otherwise.
\end{dfn}

\begin{ex}
\label{bioBex1}
\textbf{(a)}
Let $X$ be a $1$-dimensional oriented Riemannian manifold. As $\Spin(1)=\{\pm1\}\subset\Cl_1=\C,$ a spin structure is a double covering $P_{\Spin}(X)\to X.$ Also $\De_1=\C$ with $\rho_{\De_1}(e_1)=i$ and $\Si_1=\C$ is the sign representation. Let $t$ be a local coordinate on $X$ where $P_{\Spin}(X)$ is trivial. Locally, a spinor is then a complex-valued function on which the full Dirac operator acts by $\slashed{D}_X=i\frac{\d}{\d t}.$ Since $\Si_1^+=\R,$ $\Si_1^-=i\R$ we can locally identify positive spinors with real-valued functions and negative spinors with purely imaginary functions. The positive Dirac operator is $\slashed{D}_X^+=i\frac{\d}{\d t}$ and the skew-adjoint Dirac operator is $\slashed{D}_X^\skew=\frac{\d}{\d t}.$
\smallskip

\noindent
\textbf{(b)}
Let $X$ be a Riemann surface. We have $\De_2=\H$ where $\rho_{\De_2}(e_1)(q)\ab=qi,$ $\rho_{\De_2}(e_2)(q)=qj,$ and the quaternionic structure on $\De_2$ is $I(q)=iq$ and $J(q)=jq.$ Fix a complex coordinate $z=x+iy$ on $X$ where $P_{\Spin}(X)$ is trivial. Locally, a spinor is then an $\H$-valued function $\psi(z)$ on which the full Dirac operator acts by $\slashed{D}_X(\psi)=\frac{\6\psi}{\6x}i+\frac{\6\psi}{\6y}j.$ The spinor representation $\Si_2$ splits as $\Si_2^\pm=\C(1\pm j)\subset\H,$ so half-spinors $\psi^\pm(z)$ are locally complex-valued functions $f^\pm(z)$ 
and the skew-adjoint Dirac operator is the Cauchy--Riemann operator $\slashed{D}_X^\skew(f)=2\ol{\6 f/\6 \bar z}.$
\end{ex}

\addsubsection{Dirac operators in adjacent dimensions}
\label{bioB3}

Finally, we can explain the relationship between real Dirac operators in adjacent dimensions. The result is summarized in Table~\ref{bioBtab2}.

\begin{prop}
\label{bioBprop1}
\hangindent\leftmargini
\textup{\bf(a)}\hskip\labelsep
If\/ $n\equiv0,$ there is an isomorphism\/ $\Phi_{n+1}\colon\Si_n\ot\C\to\Si_{n+1}|_{\Spin(n)}$ of complex representations of\/ $\Spin(n)$ under which the real structure\/ $\si_{n+1}=\id_{\Si_{n+1}^+}\op-\id_{\Si_{n+1}^-}$ on\/ $\Si_{n+1}=\Si_{n+1}^+\op\Si_{n+1}^-$ corresponds to the complex conjugation on\/ $\Si_n\ot\C$ and the Clifford multiplication to
\begin{equation}
\begin{aligned}
 \rho_{\De_{n+1}}(e_k)\circ\Phi_{n+1}&=\Phi_{n+1}\circ(\rho_{\De_n}(\om_n e_k)\ot i),&&(1\le k\le n)\\
 \rho_{\De_{n+1}}(e_{n+1})\circ\Phi_{n+1}&=\Phi_{n+1}\circ(\rho_{\De_n}(\om_n)\ot i).
\end{aligned}
\label{bioBeq4}
\end{equation}
\begin{enumerate}
\stepcounter{enumi}
\item
If\/ $n\equiv1,$ there is a\/ $\C$-linear isomorphism\/ $\Phi_{n+1}\colon\Si_n\op\Si_n\to\Si_{n+1}|_{\Spin(n)}$ under which the Clifford multiplication and quaternionic structure are
\begin{equation}
\begin{aligned}
	\rho_{\De_{n+1}}(e_k)\circ\Phi_{n+1}&=\Phi_{n+1}\circ
	\begin{pmatrix}
	\rho_{\De_n}(e_k) & 0\\
	0 & -\rho_{\De_n}(e_k)
	\end{pmatrix}, &&\mathrlap{(1\le k\le n)}\\
	\rho_{\De_{n+1}}(e_{n+1})\circ\Phi_{n+1}&=\Phi_{n+1}\circ
	\begin{pmatrix}
	0 & \id_{\De_n}\\
	-\id_{\De_n} & 0
	\end{pmatrix},\\
	J\circ\Phi_{n+1}&=\Phi_{n+1}\circ
	\begin{pmatrix}
		0 & \si_n\\
		-\si_n & 0
	\end{pmatrix},
\end{aligned}
\label{bioBeq5}
\end{equation}
where\/ $\si_n$ is the real structure on\/ $\Si_n$ from \textup{(a)}.
\item
If\/ $n\equiv2,$ there is an\/ $\H$-linear isomorphism\/ $\Phi_{n+1}\colon\Si_n\to\Si_{n+1}|_{\Spin(n)}$ with
\begin{equation}
\begin{aligned}
 \rho_{\De_{n+1}}^+(e_k)\circ\Phi_{n+1} &=\Phi_{n+1}\circ\rho_{\De_n}(e_k), &&{(1\le k\le n)}\\
 \rho_{\De_{n+1}}^+(e_{n+1})\circ\Phi_{n+1} &=-\Phi_{n+1}\circ\rho_{\De_n}(\om_n).
\end{aligned}
\label{bioBeq6}
\end{equation}
\item
If\/ $n\equiv 3$ let\/ $\K=\H,$ and if\/ $n\equiv 7$ let\/ $\K=\R.$ There is a pair of\/ $\K$-linear isomorphisms\/ $\Phi_{n+1}^\pm\colon \Si_n\to\Si_{n+1}^\pm|_{\Spin(n)}$ under which
\begin{equation}
\begin{aligned}
 \rho_{\De_{n+1}}(e_k)\circ\Phi_{n+1}^\pm &=\Phi_{n+1}^\mp\circ\rho_{\De_n^+}(e_k),
 &&(1\le k\le n)\\
 \rho_{\De_{n+1}}(e_{n+1})\circ\Phi_{n+1}^\pm &= \pm\Phi_{n+1}^\mp.
\end{aligned}
\label{bioBeq7}
\end{equation}
\item
If\/ $n\equiv4,$ there is an\/ $\H$-linear isomorphism\/ $\Phi_{n+1}\colon\Si_n\to\Si_{n+1}|_{\Spin(n)}$ with
\begin{equation}
\begin{aligned}
\rho_{\De_{n+1}}(e_k)\circ\Phi_{n+1}&=\Phi_{n+1}\circ\rho_{\De_n}(e_k),&&(1\le k\le n)\\
\rho_{\De_{n+1}}(e_{n+1})\circ\Phi_{n+1}&=I\circ\Phi_{n+1}\circ\rho_{\De_n}(\om_n).
\end{aligned}
\label{bioBeq8}
\end{equation}
\item
If\/ $n\equiv5,$ then\/ $\Si_n$ has a quaternionic structure\/ $I,$ $J$ and there is an isomorphism of representations\/ $\Phi_{n+1}\colon\Si_n\to\Si_{n+1}|_{\Spin(n)}$ with
\begin{equation}
\begin{aligned}
 \rho_{\De_{n+1}}(e_ke_{n+1})\circ\Phi_{n+1}&=\Phi_{n+1}\circ\rho_{\De_n}(e_k),&&(1\le k\le n)
\\
\rho_{\De_{n+1}}(\om_{n+1})\circ\Phi_{n+1}&=\Phi_{n+1}\circ I,\\
\rho_{\De_{n+1}}(e_{n+1})\circ\Phi_{n+1}&=\Phi_{n+1}\circ J.
\end{aligned}
\label{bioBeq9}
\end{equation}
\item
If\/ $n\equiv6,$ there is an\/ $\R$-linear isomorphism\/ $\Phi_{n+1}\colon\Si_n\to\Si_{n+1}|_{\Spin(n)}$ with
\begin{equation}
\begin{aligned}
\rho_{\De_{n+1}^+}(e_k)\circ\Phi_{n+1}&=\Phi_{n+1}\circ\rho_{\De_n}(e_k),&&(1\le k\le n)\\
\rho_{\De_{n+1}^+}(e_{n+1})\circ\Phi_{n+1}&=\Phi_{n+1}\circ I.
\end{aligned}
\label{bioBeq10}
\end{equation}
\end{enumerate}
The isomorphisms \textup{(a)--(g)} are orthogonal and unique up to scalar factors.
\end{prop}

\begin{proof}
We discuss (a) in detail and give briefer arguments for (b)--(g).
\smallskip

\noindent
\textbf{(a)} If $n\equiv 0,$ then \eqref{bioBeq1} implies that the $\C$-linear endomorphisms of $\De_n\ot_\R\C$ defined by $E_k=\rho_{\De_n}(\om_n e_k)\ot i,$ $1\le k\le n,$ and $E_{n+1}=\rho_{\De_n}(\om_n)\ot i$ satisfy the Clifford relations for $\Cl_{n+1}.$ Since $\De_{n+1}$ is the unique irreducible complex representation of $\Cl_{n+1}$ and $\dim_\C\De_n\ot_\R\C=\dim_\C\De_{n+1},$ there exists a $\C$-linear isomorphism $\Phi_{n+1}\colon\De_n\ot_\R\C\to\De_{n+1},$ orthogonal and unique up to a scalar.

According to Theorem~\ref{bioBthm1} there is a real structure on $\De_{n+1}|_{\Cl_{n+1}^0}$ (the irreducible representations of the algebra $\Cl_n\cong\Cl_{n+1}^0$ are real), unique up to a sign, which is fixed by the choice of splitting into $\Si_{n+1}^\pm.$ On $\De_n\ot_\R\C$ the complex conjugation defines a real structure. By Schur's Lemma, the real structures are automatically preserved up to sign, so by passing to $i\Phi_{n+1}$ if necessary we can arrange for them to be preserved exactly. Since $\Si_n=\De_n|_{\Spin(n)}$ and $\Si_{n+1}|_{\Spin(n)}=\De_{n+1}|_{\Spin(n)},$ the map $\Phi_{n+1}$ defines the claimed isomorphism of representations satisfying \eqref{bioBeq4}.
\smallskip

\noindent
\textbf{(b)} If $n\equiv 1,$ the complex representation $\De_n\op\al^*(\De_n)$ of $\Cl_n$ is extended by
\[
 E_{n+1}=
 \begin{pmatrix}
 	0 & \id_{\De_n}\\
 	-\id_{\De_n} & 0
 \end{pmatrix},
 \quad
 J=
 \begin{pmatrix}
 	0 & \si_n\\
 	-\si_n & 0
 \end{pmatrix}
\]
to a quaternionic representation of $\Cl_{n+1},$ which is thus isomorphic to $\De_{n+1}.$
\smallskip

\noindent
\textbf{(c)} If $n\equiv 2,$ the $\H$-linear endomorphisms of $\De_n$ defined by
\[
 E_k=\rho_{\De_n}(e_k), 1\le k\le n,\quad E_{n+1}=-\rho_{\De_n}(\om_n)
\]
satisfy the Clifford relations for $\Cl_{n+1}.$ Since $E_1\cdots E_n E_{n+1}=1,$ there is an $\H$-linear isomorphism $\De_n\to\De_{n+1}^+$ as in \eqref{bioBeq6}.
\smallskip

\noindent
\textbf{(d)} If $n\equiv 3,7,$ the $\K$-linear endomorphisms of $\De_n^+\op\De_n^+$ defined by
\[
 E_k=
 \begin{pmatrix}
 \rho_{\De_n^+}(e_k) & 0\\
 0 & -\rho_{\De_n^+}(e_k)
 \end{pmatrix},
\quad
(1\le k\le n)
\quad
 E_{n+1}=
 \begin{pmatrix}
 0 & \id_{\De_n^+}\\
 -\id_{\De_n^+} & 0
 \end{pmatrix}.
\]
satisfy the Clifford relations for $\Cl_{n+1}.$ It follows that there is a $\K$-linear isomorphism $\Psi_{n+1}\colon\De_n^+\op\De_n^+\to\De_{n+1}$ taking $E_k$ to $\rho_{\De_{n+1}}(e_k)$ for $1\le k\le n+1.$ Now
\[
 E_1\cdots E_nE_{n+1}=
 \begin{pmatrix}
 0 & \om_n\\
 \om_n & 0
 \end{pmatrix}
 =
 \begin{pmatrix}
 0 & \id_{\De_n^+}\\
 \id_{\De_n^+} & 0
 \end{pmatrix}
 \quad\text{as $\om_n|_{\De_n^+}=\id_{\De_n^+},$}
\]
so the splitting of $\Si_{n+1}=\De_{n+1}|_{\Spin(n+1)}$ into the $(\pm1)$-eigenspaces $\Si_{n+1}^\pm$ of the operator $\rho_{\De_{n+1}}(\om_{n+1})$ corresponds to the splitting of $\De_n^+\op\De_n^+$ into the diagonal and anti-diagonal spaces. Now define $\Phi_{n+1}^\pm(\phi)=\Psi_{n+1}(\phi,\pm\phi).$
\smallskip

\noindent
\textbf{(e)} If $n\equiv 4,$ the $\H$-linear endomorphisms on $\De_n$ defined by $E_k=\rho_{\De_n}(e_k)$ for $1\le k\le n$ and $E_{n+1}=I\rho_{\De_n}(\om_n)$ satisfy the Clifford relations for $\Cl_{n+1},$ so there is an $\H$-linear isomorphism $\Phi_{n+1}\colon\De_n\to\De_{n+1}.$
\smallskip

\noindent
\textbf{(f)} If $n\equiv 5,$ consider the real representation $j_n^*(\De_{n+1})$ of $\Cl_n,$ where $e_k$ acts by $\rho_{\De_{n+1}}(e_ke_{n+1})$ for $1\le k\le n.$ Observe that $I'=\rho_{\De_{n+1}}(\om_{n+1})$ and $J'=\rho_{\De_{n+1}}(e_{n+1})$ anti-commute with the Clifford multiplication by vectors, hence define a quaternionic structure on the representation $j_n^*(\De_{n+1})|_{\Cl_n^0}.$ The unique irreducible quaternionic representation $\De_n|_{\Cl_n^0}$ has the same dimension, hence there exists an $\H$-linear isomorphism $\Phi_{n+1}\colon\De_n|_{\Cl_n^0}\to j_n^*(\De_{n+1})|_{\Cl_n^0}.$
\smallskip

\noindent
\textbf{(g)} If $n\equiv 6,$ the $\R$-linear endomorphisms of $\De_n$ defined by $E_k=\rho_{\De_n}(e_k),$ $1\le k\le n,$ and $E_{n+1}=-\rho_{\De_n}(\om_n)$ satisfy the Clifford relations for $\Cl_{n+1}.$ Moreover, $E_1\cdots E_n E_{n+1}=1,$ so there is an $\R$-linear isomorphism $\Phi_{n+1}\colon\De_n\to\De_{n+1}^+.$ As $\om_n^2=-1$ and $\om_n$ commutes with $\Cl_n^0,$ we can arrange for $I=\rho_{\De_n}(\om_n).$
\end{proof}

We now compute the Dirac operator over $Y=X\t\R$ in a local orthonormal frame $e_1,\ldots, e_n$ of $X$ and where $e_{n+1}=\6/\6 t.$ We have
\[
 P_\Spin(TY)=(P_\Spin(TX)\t\Spin(n+1))/\Spin(n)\t \R
\]
and hence the isomorphisms of representations from Proposition~\ref{bioBprop1} induce isomorphisms $\Phi_{n+1}^{(\pm)}$ between the spinor bundle of $Y$ and the pullback of the spinor bundle of $X$ to $X\t\R.$

\begin{table}[htb]
\centerline{\begin{tabular}{|l|m{2.6cm}|m{3cm}|m{4.4cm}|}
\hline
$\bs n\equiv$ & \bf Spinor bundles & \bf Dirac on $X^n$ & \bf On cylinder $X\t\R$\\\hline
$0$ & $\Si_X^\pm$ real & $\slashed{D}_X^+$ $\R$-linear & $\slashed{D}_Y^\skew=
\begin{pmatrix}
\6/\6 t & \slashed{D}_X^-\\
-\slashed{D}_X^+ & -\6/\6 t
\end{pmatrix}
$ \\\hline
$1$ & $\Si_X^+\cong\Si_X^-$\newline real & $\slashed{D}_X^\skew$ $\R$-linear\newline skew-adjoint & $\slashed{D}_Y^\skew=\ol{\frac{\6}{\6 t}+i(\slashed{D}_X^\skew)_\C}$\\\hline
$2$ & \raisebox{-4pt}{$\Si_X^+\cong\ol{\Si_X^-}$}\newline complex & $\slashed{D}_X^\skew$ $\C$-linear\newline skew-adjoint & $\slashed{D}_Y=\begin{pmatrix}
-i\6/\6 t & -\ol{\slashed{D}_X^\skew}\\
\slashed{D}_X^\skew & i\6/\6 t
\end{pmatrix}
$\\\hline
$3$ & $\Si_X$ quat. & $\slashed{D}_X$ $\H$-linear\newline self-adjoint & $\slashed{D}_Y^+=\frac{\6}{\6 t}+\slashed{D}_X$\\\hline
$4$ & $\Si_X^\pm$ quat. & $\slashed{D}_X^+$ $\H$-linear & $\slashed{D}_Y=\begin{pmatrix}
i\6/\6 t	& \slashed{D}_X^-\\
\slashed{D}_X^+ & -i\6/\6 t
\end{pmatrix}
$\\\hline
$5$ & $\Si_X$ quat. & $\slashed{D}_X$ $\C$-linear\newline self-adjoint (anti-\newline commutes with $J$)& $\slashed{D}_Y=J\bigl(\frac{\6}{\6 t}+\slashed{D}_X\bigr)$\\\hline
$6$ & $\Si_X$ complex & $\slashed{D}_X$ $\C$-anti-linear\newline self-adjoint & $\slashed{D}_Y=I\frac{\6}{\6 t}+\slashed{D}_X$\\\hline
$7$ & $\Si_X$ real & $\slashed{D}_X$ $\R$-linear\newline self-adjoint & $\slashed{D}_Y^+=\frac{\6}{\6 t}+\slashed{D}_X$\\\hline
\end{tabular}}
\caption{Dirac operators on a spin $n$-manifold $X$ and over a collar neighborhood of a spin $(n+1)$-manifold $Y$ which can be identified with a cylinder $X\t[0,\varepsilon)\subset X\t\R,$ up to deformation.}
\label{bioBtab2}
\end{table}

\subsubsection{Case $n\equiv0$}
\label{bioB31}

The map $\Phi_{n+1}$ from Proposition~\ref{bioBprop1}(a) restricts to an isomorphism $\Phi_{n+1}^+\colon\Si_n\to\Si_{n+1}^+$ of the real points. Recall that $\Si_n$ decomposes into the $(\pm1)$-eigenspaces $\Si_n^\pm$ of $\om_n.$ With respect to this decomposition, we have
\ea*
 \slashed{D}_{X\t \R}^\skew&\circ\Phi_{n+1}^+=-I\left(\sum_{k=1}^n\rho_{\De_{n+1}}(e_k)\na_k + \rho_{\De_{n+1}}(e_{n+1})\frac{\6}{\6 t}\right)\circ\Phi_{n+1}^+&& \text{by \eqref{bioBeq2}}\\
 &=\Phi_{n+1}^+\circ\left(\sum_{k=1}^n
 \begin{pmatrix}
 0 & \rho_{\De_n}(e_k)\\
 -\rho_{\De_n}(e_k) & 0
 \end{pmatrix}
 \na_k+
 \begin{pmatrix}
 1 & 0\\
 0 & -1
 \end{pmatrix}
 \frac{\6}{\6 t}\right)&& \text{by \eqref{bioBeq4}}\\
 &=\Phi_{n+1}^+\circ
 \begin{pmatrix}
 	\frac{\6}{\6 t} & \slashed{D}_X^-\\
 	-\slashed{D}_X^+ & -\frac{\6}{\6 t}
 \end{pmatrix}.
\ea*

\subsubsection{Case $n\equiv 1$}
\label{bioB32}

The splitting of $\Si_{n+1}$ into the $(\pm i)$-eigenspaces $\Si_{n+1}^\pm$ of $\om_{n+1}$ corresponds under the isomorphism $\Phi_{n+1}$ from Proposition~\ref{bioBprop1}(b) to the splitting of $\Si_n\op\Si_n$ into the diagonal and anti-diagonal subspaces because
\[
 \rho_{\De_{n+1}}(\om_{n+1})\circ\Phi_{n+1}=\Phi_{n+1}\circ\begin{pmatrix}0 & I\\ I & 0\end{pmatrix}
\]
by \eqref{bioBeq5}. In particular, $\Phi_{n+1}$ induces isomorphisms $\Phi_{n+1}^\pm\colon\Si_n\to\Si_{n+1}^\pm,$ $\phi\mapsto\Phi_{n+1}(\phi,\pm\phi).$ Using \eqref{bioBeq3} and \eqref{bioBeq5} we now compute
\ea*
 \slashed{D}_{X\t \R}^\skew\circ&\Phi_{n+1}^+(\phi)=-J\left(\sum_{k=1}^n \rho_{\De_{n+1}}(e_k)\na_k
 +\rho_{\De_{n+1}}(e_{n+1})\frac{\6}{\6 t}\right)
 \circ\Phi_{n+1}\begin{pmatrix}
 \phi\\
 \phi
 \end{pmatrix}
\\
 &=
 -J\circ\Phi_{n+1}\left(
 \sum_{k=1}^n
 \begin{pmatrix}
 e_k & 0\\
 0 & -e_k	
 \end{pmatrix}
 \begin{pmatrix}\na_k\phi\\\na_k\phi\end{pmatrix}
 +
 \begin{pmatrix}
 	0 & 1\\
 	-1 & 0
 \end{pmatrix}
 \begin{pmatrix}\6\phi/\6 t \\ \6\phi/\6 t\end{pmatrix}
 \right)\\
 &=\Phi_{n+1}
 \begin{pmatrix}
 	0 & -\si_n\\
 	\si_n & 0
 \end{pmatrix}
 \begin{pmatrix}
 	\sum_{k=1}^n e_k\na_k\phi+\6\phi/\6 t\\
 	\sum_{k=1}^n -e_k\na_k\phi-\6\phi/\6 t
 \end{pmatrix}\\
 &=\Phi_{n+1}^+\circ\si_n(\slashed{D}_X+\6/\6 t)\phi
\ea*
Splitting $\Si_n$ further into the eigenspaces $\Si_n^\pm$ of $\si_n,$ we can write this as
\[
 \si_n(\slashed{D}_X+\6/\6 t)=\begin{pmatrix}
	1 & 0\\
	0 & -1
\end{pmatrix}
\begin{pmatrix}
	\6/\6 t & \slashed{D}_X^-\\
	\slashed{D}_X^+ & \6/\6 t
\end{pmatrix}
=
\ol{\frac{\6}{\6 t}+i(\slashed{D}_X^\skew)_\C}.
\]

\subsubsection{Case $n\equiv 2$}
\label{bioB33}

We have
\ea*
 &\slashed{D}_{X\t\R}\circ\Phi_{n+1}=\left(\sum_{k=1}^n\rho_{\De_{n+1}^+}(e_k)\na_k+\rho_{\De_{n+1}^+}(e_{n+1})\frac{\6}{\6 t}\right)\circ\Phi_{n+1}\\
 &=\Phi_{n+1}\circ\left(\sum_{k=1}^n\rho_{\De_n}(e_k)\na_k-\rho_{\De_n}(\om_n)\frac{\6}{\6 t}\right)=\Phi_{n+1}\circ\begin{pmatrix}
	-i\6/\6 t & \slashed{D}_X^-\\
	\slashed{D}_X^+ & i\6/\6 t
\end{pmatrix},
\ea*
where at the last step we split $\Si_n$ into the $(\pm i)$-eigenspaces $\Si_n^\pm$ of $\om_n.$ The composition of $\Phi_{n+1}\colon\Si_n\to\Si_{n+1}$ with $\id_{\Si_n^+}\op J\colon\Si_n^+\op\ol{\Si_n^+}\to\Si_n^+\op\Si_n^-$ is an isomorphism $\Psi_{n+1}\colon\Si_n^+\op\ol{\Si_n^+}\to\Si_{n+1}.$ The previous calculation and \eqref{bioBeq2} imply
\[
 \slashed{D}_{X\t\R}\circ\Psi_{n+1}=\Psi_{n+1}\circ\begin{pmatrix}
	-i\6/\6 t & -\ol{\slashed{D}_X^\skew}\\
	\slashed{D}_X^\skew & i\6/\6 t
\end{pmatrix}
\]
(the `$i$' in lower right corner stands for the scalar multiplication in $\ol{\Si_n^+},$ hence is equal to the action of $-i$ on $\Si_n^+$).

\subsubsection{Case $n\equiv 3,7$}
\label{bioB34}

Using \eqref{bioBeq7} we find
\ea*
 \slashed{D}_{X\t\R}^+\circ\Phi_{n+1}^+&=\left(\sum_{k=1}^n\rho_{\De_{n+1}}(e_k)\na_k+\rho_{\De_{n+1}}(e_{n+1})\frac{\6}{\6 t}\right)\circ\Phi_{n+1}^+\\
 &=\Phi_{n+1}^-\circ
 \left(
 \sum_{k=1}^n \rho_{\De_n}(e_k)\na_k+\frac{\6}{\6 t}
 \right)=\Phi_{n+1}^-\circ\left(\slashed{D}_X+\frac{\6}{\6 t}\right).
\ea*

\subsubsection{Case $n\equiv 4$}
\label{bioB35}

We have
\ea*
 \slashed{D}_{X\t\R}\circ\Phi_{n+1}&=\left(\sum_{k=1}^n\rho_{\De_{n+1}}(e_k)\na_k+\rho_{\De_{n+1}}(e_{n+1})\frac{\6}{\6 t}\right)\circ\Phi_{n+1}\\
 &=\Phi_{n+1}\circ\left(\rho_{\De_n}(e_k)\na_k+\rho_{\De_n}(\om_n)\frac{\6}{\6 t}\right)&&\text{by \eqref{bioBeq8}}\\
 &=\Phi_{n+1}\circ\begin{pmatrix}
 i\6/\6t & \slashed{D}_X^-\\
 \slashed{D}_X^+ & -i\6/\6t
 \end{pmatrix},
\ea*
where at the last step we split $\Si_n$ into the $(\pm i)$-eigenspaces $\Si_n^\pm$ of $\om_n.$

\subsubsection{Case $n\equiv 5$}
\label{bioB6}

We have
\ea*
 &\slashed{D}_{X\t\R}\circ\Phi_{n+1}=
 \left(\sum_{k=1}^n \rho_{\De_{n+1}}(e_k)\na_k
 +\rho_{\De_{n+1}}(e_{n+1})\frac{\6}{\6 t}\right)
 \mathrlap{\circ\Phi_{n+1}}\\
 &=\rho_{\De_{n+1}}(e_{n+1})\left(\sum_{k=1}^n\rho_{\De_{n+1}}(e_ke_{n+1})\na_k+\frac{\6}{\6 t}\right)\circ\Phi_{n+1}\\
  &=\rho_{\De_{n+1}}(e_{n+1})\circ\Phi_{n+1}\circ\left(\sum_{k=1}^n \rho_{\De_n}(e_k)\na_k+\frac{\6}{\6 t}\right)&&\text{by \eqref{bioBeq9}}\\
 &=\Phi_{n+1}\circ J \left(\sum_{k=1}^n\rho_{\De_{n+1}}(e_k)\na_k+\frac{\6}{\6 t}\right)&&\text{by \eqref{bioBeq9}}\\
 &=\Phi_{n+1}\circ J(\slashed{D}_X+\6/\6 t).
\ea*

\subsubsection{Case $n\equiv 6$}
\label{bioB7}

Using \eqref{bioBeq10} we find
\ea*
 &\slashed{D}_{X\t\R}\circ\Phi_{n+1}=
 \left(\sum_{k=1}^n \rho_{\De_{n+1}^+}(e_k)\na_k
 +\rho_{\De_{n+1}^+}(e_{n+1})\frac{\6}{\6 t}\right)
 \circ\Phi_{n+1}\\
 &=\Phi_{n+1}\circ\left(\sum_{k=1}^n \rho_{\De_n}(e_k)\na_k
 +I\frac{\6}{\6 t}\right)=\Phi_{n+1}\circ\left(\slashed{D}_X + I\frac{\6}{\6 t}\right).
 \ea*


\begin{thebibliography}{99}

\bibitem{Aro} N. Aronszajn, {\it A unique continuation theorem for solutions of elliptic partial differential equations or inequalities of second order}, J. Math. Pures Appl. 36 (1957), 235--249.

\bibitem{AtBoSh} M.F. Atiyah, R. Bott and A. Shapiro, {\it Clifford modules}, Topology 3 (1964), 3--38.

\bibitem{AtPaSi1} M.F. Atiyah, V.K. Patodi and I.M. Singer, {\it Spectral asymmetry and Riemannian geometry.~I}, Math. Proc. Cambridge Philos. Soc. 77 (1975), 43--69.

\bibitem{AtSi4} M.F. Atiyah and I.M. Singer, {\it The Index of Elliptic Operators: IV}, Ann. of Math. 92 (1970), 119--138.

\bibitem{AtSi5} M.F. Atiyah and I.M. Singer, {\it The Index of Elliptic Operators: V}, Ann. of Math. 93 (1971), 139--149.

\bibitem{BaBa1} W. Ballmann and C. B\"ar, {\it Boundary value problems for elliptic differential operators of first order}, pp. 1--78 in Surv. Differ. Geom. 17, Int. Press, Boston, MA, 2012. \href{http://arxiv.org/abs/1101.1196}{arXiv:1101.1196}

\bibitem{Do} S.K. Donaldson, {\it Floer Homology Groups in Yang-Mills Theory}, C.U.P., Cambridge, 2009.

\bibitem{DoKr} S.K. Donaldson and P.B. Kronheimer, {\it The Geometry of Four-Manifolds}, O.U.P., 1990.

\bibitem{DoSe} S.K. Donaldson and E. Segal, {\it Gauge Theory in Higher Dimensions, II}, Surveys in Diff. Geom. 16 (2011), 1--41. \href{http://arxiv.org/abs/0902.3239}{arXiv:0902.3239}

\bibitem{DoTh} S.K. Donaldson and R.P. Thomas, {\it Gauge Theory in Higher Dimensions}, Chapter 3 in S.A. Huggett et al., editors, {\it The Geometric Universe}, O.U.P., Oxford, 1998.

\bibitem{Joyc5} D. Joyce, {\it Conjectures on counting associative $3$-folds in $G_2$-manifolds}, pp. 97--160 in V. Mu\~noz et al., editors, {\it Modern Geometry: A Celebration of the Work of Simon Donaldson}, Proc. Symp. Pure Math. 99, A.M.S., Providence, RI, 2018. \href{http://arxiv.org/abs/1610.09836}{arXiv:1610.09836}

\bibitem{JTU} D. Joyce, Y. Tanaka and M. Upmeier, {\it On orientations for gauge-theoretic moduli spaces}, Adv. Math. 362 (2020). \href{https://arxiv.org/abs/1811.01096}{arXiv:1811.01096}

\bibitem{JoUp1} D. Joyce and M. Upmeier, {\it Canonical orientations for moduli spaces of\/ $G_2$-instantons with gauge group $\SU(m)$ or $\U(m)$}, J. Differential Geom. 124 (2023), 199--229. \href{http://arxiv.org/abs/1811.02405}{arXiv:1811.02405}

\bibitem{JUother} D. Joyce and M. Upmeier, {\it Bordism categories and orientations of gauge theory moduli spaces}, in preparation, 2023.

\bibitem{JUfurther} D. Joyce and M. Upmeier, {\it Flag structures and moduli spaces of calibrated submanifolds}, in preparation, 2024.

\bibitem{LaMi} H.B. Lawson and M.-L. Michelsohn, {\it Spin geometry}, Princeton Math. Series 38, P.U.P., Princeton, NJ, 1989.

\bibitem{LoMc} R.B. Lockhart and R.C. McOwen, {\it Elliptic differential operators on noncompact manifolds}, Ann. Scuola Norm. Sup. Pisa (1987), 409--447.

\bibitem{MacL} S. MacLane, {\it Categories for the working mathematician}, Springer, New York, 1971.

\bibitem{Mein13} E. Meinrenken, {\it Clifford algebras and Lie theory}, Ergeb. Math. Grenzgeb. 58, Springer, Heidelberg, 2013.

\bibitem{Sinh} H.X. Sinh, {\it Gr-cat\'egories}, Th\'ese de Doctorat, Universit\'e Paris VII, 1975.

\bibitem{Upme} M. Upmeier, {\it A categorified excision principle for elliptic symbol families}, Q. J. Math.~72 (2021), 1099--1132. \href{http://arxiv.org/abs/1901.10818}{arXiv:1901.10818}.

\bibitem{Walp1} T. Walpuski, {\it Gauge theory on $G_2$-manifolds}, PhD Thesis, Imperial College London, 2013.

\end{thebibliography}
\end{document}